\documentclass[3p]{elsarticle}

\usepackage{lineno,hyperref}
\usepackage{moreverb}
\usepackage{amsmath,amsthm}
\usepackage{color,xcolor}
\usepackage{amsfonts}

\usepackage{mathrsfs}
\usepackage{amssymb}
\usepackage{hyperref}
\usepackage{graphicx}
\usepackage{multicol}
\usepackage{bm}
\usepackage{appendix}

\newtheorem{theorem}{Theorem}[section]
\newtheorem{definition}[theorem]{Definition} %% Definitions are numbered with theorems together.
\newtheorem{lemma}[theorem]{Lemma}

\newtheorem{remark}[theorem]{Remark}

\renewcommand{\geq}{\geqslant}
\renewcommand{\leq}{\leqslant}

\newcommand{\dd}{\operatorname{d}\! }
\newcommand{\dt}{\operatorname{d}\! t}
\newcommand{\ds}{\operatorname{d}\! s}

\newcommand{\dw}{\operatorname{d}\! W}

\newcommand{\dtp}{\dt\otimes\mathrm{d}\BP}

\newcommand{\BE}{\mathbb{E}}

\newcommand{\BR}{\mathbb{R}}
\newcommand{\BP}{\mathbb{P}}

\newcommand{\essinf}{\ensuremath{\mathrm{ess\:inf\:}}}
\newcommand{\esssup}{\ensuremath{\mathrm{ess\:sup\:}}}

\newcommand{\seta}{\mathscr{A}}

\begin{document}

\begin{frontmatter}

\title{Competitive optimal portfolio selection in a non-Markovian financial market:
A  backward stochastic differential equation study
% \tnoteref{mytitlenote}
}
\author[1]{Guangchen Wang}
\ead[url]{wguanchen@sdu.edu.cn}

\author[2]{Zuo Quan Xu}
\ead[url]{maxu@polyu.edu.hk}

\author[1,2]{Panpan Zhang}
\ead[url]{zhangpanpan@mail.sdu.edu.cn}

\address[1]{School of Control Science and Engineering, Shandong University, Jinan 250061, China}
\address[2]{Department of Applied Mathematics, The Hong Kong Polytechnic University, Hong Kong, China}

\begin{abstract}
This paper studies a competitive optimal portfolio selection problem in a model where
the interest rate, the appreciation rate and volatility rate of the risky asset are all stochastic processes, thus forming a non-Markovian financial market.
In our model, all investors (or agents) aim to obtain an above-average wealth at the end of the common investment horizon.
This competitive optimal portfolio problem is indeed a non-zero stochastic differential game problem. The quadratic BSDE theory is applied to tackle the problem and Nash equilibria in suitable spaces are found.
We discuss both the CARA and CRRA utility cases. For the CARA utility case,
there are three possible scenarios depending on market and competition parameters: a unique Nash equilibrium, no Nash equilibrium, and infinite Nash equilibria.
The Nash equilibrium is given by the solutions of a quadratic BSDE and a linear BSDE with unbounded coefficient when it is unique.
Different from the wealth-independent Nash equilibria in the existing literature, the equilibrium in our paper is of feedback form of wealth.
For the CRRA utility case, the issue is a bit more complicated than the CARA utility case.
We prove the solvability of a new kind of quadratic BSDEs with unbounded coefficients.
A decoupling technology is used to relate the Nash equilibrium to a series of 1-dimensional quadratic BSDEs.
With the help of this decoupling technology, we can even give the limiting strategies for both cases when the number of agent tends to be infinite.
\end{abstract}


\end{frontmatter}

\section{Introduction }

Nobel laureate R. C. Merton first proposed a continuous time model for optimal portfolio and consumption problem in his pioneering work \cite{Merton}, which is actually a stochastic optimal control problem. Lacker and Zariphopoulou extended Merton's work to relative performance criteria case in \cite{Zariphopoulou 1}. In their work, fund managers (or agents) trading between a common risk-less bond and an individual stock in a Markovian financial market, where the market parameters are all constant. All agents share a common investment period and they aim to maximize their expected utility of above-average wealth at the end of this period. This competitive optimal portfolio problem is indeed a non-zero stochastic differential game problem. In \cite{Zariphopoulou 1}, authors gave an explicit constant equilibrium strategy for the game, and showed that it is unique in the class of constant equilibria. Later, Lacker and Soret \cite{Lacker2} expanded the model of \cite{Zariphopoulou 1} to a situation involving competitive consumption. Competitive optimal portfolio problems have been extensively studied in the past few years, see, e.g., \cite{MF}, \cite{Huang}, \cite{AF}.

Evidently, financial markets are mostly non-Markovian in practice. For instance, the market parameters, such as the interest rate, stock
appreciation rates and volatilities, are affected by noises caused by various factors such as global or local
politics, economic growth rate, legal, military, corporate governance. Hence, it is necessary to allow the
market parameters to be stochastic processes. In a non-Markovian financial market, the exponential utility, also called the constant absolute risk aversion (CARA) utility, has been considered in \cite{Zariphopoulou 2} and \cite{FU Guangxing}.
Authors studied the competitive optimal portfolio problem in the CARA utility case with a partially observable stock appreciation rate in
\cite{Zhou Chao} and \cite{Panpan}.
Because the interest rate is assumed to be a constant, the Nash equilibria in \cite{Zariphopoulou 2}, \cite{FU Guangxing}, \cite{Zhou Chao} and \cite{Panpan} are all wealth-independent.
For constant relative risk aversion (CRRA) utility (that is, power or logarithmic utility), the literature is sparse. Among the rare results,
in \cite{JEDC}, the competitive optimal portfolio problem for two investors and power utility is studied by Hamilton-Jacobi-Bellman-Isaacs (HJBI) equations. Hu et al. \cite{Hu AAP2}, \cite{Xu arXiv2} investigated optimal portfolio problems for exponential, power and logarithmic utility by backward stochastic differential equations (BSDEs).

In this paper, we consider a competitive optimal portfolio problem for exponential, power and logarithmic utility in a non-Markovian financial market. We use the framework and notions of \cite{Zariphopoulou 1}, but allow for a more general non-Markovian market model (beyond the log-normal case).
In our model, the interest rate, the appreciation rate and volatility rate of the risky stock are all stochastic processes.
Inspired by \cite{Xu arXiv2}, we use a BSDE method to handle this game problem.
As a powerful tool to handle optimal portfolio problems in non-Markovian markets, BSDE was first introduced by Pardoux and Peng \cite{Peng1}. El Karoui et al. \cite{Peng2} gave some applications of BSDEs in mathematical finance. Benefiting from numerous studies on the solvability of BSDEs, including \cite{Fan Hu Tang} and \cite{Xu arXiv1}, we can give the explicit expressions of Nash equilibria by the solutions of some BSDEs.

We use admissible best response (ABR) method to discuss the Nash equilibrium in a suitable space for the exponential utility case.
We prove there is a unique admissible best response in this case and equate the Nash equilibrium to the solution of a system of linear BSDEs.
There are three cases depending on the market and competition parameters: a unique Nash equilibrium, no Nash equilibrium, and infinite Nash equilibria.
In the first case, the unique Nash equilibrium is given by the solutions of a quadratic BSDE and a linear BSDE with unbounded coefficient.
Different from the Nash equilibrium in \cite{Zariphopoulou 2}, \cite{FU Guangxing}, \cite{Zhou Chao} and \cite{Panpan}, the one in our paper is a feedback of wealth.

In the CRRA utility case, we find a Nash equilibrium in a suitable space by the best response (BR) method.
This issue is a bit more complicated than the exponential utility case because the Nash equilibrium is given by a coupled $n$-dimensional quadratic BSDEs.
We use a decoupling technology to relate the Nash equilibrium to a series of 1-dimensional quadratic BSDEs.
With the help of this decoupling technology, we can even give the limiting strategy when the number of agent tends to be infinite.
To obtain the admissibility of the Nash equilibrium,
we prove the solvability of a new kind of quadratic BSDE with unbounded coefficient.

The remainder of this paper is organized as follows. In Section \ref{PF}, we formulate the competitive optimal portfolio selection problem in the CARA utility case and in the CRRA utility case in a non-Markovian financial market, and introduce two methods to obtain a Nash equilibrium.
In Section \ref{CARA}, we use the ABR method to discuss the Nash equilibrium in some space in the CARA utility case.
In Section \ref{CRRA}, we use the BR method and a decoupling technology to give a Nash equilibrium in some space in the CRRA utility case.
Finally, Section \ref{Conclusion} concludes the paper.

\subsection*{Notation}% \label{PF}

Let $(\Omega, \mathcal{F}, \{\mathcal{F}_{t}\}_{0\leq t\leq T}, \BP)$ be a fixed complete filtered probability space
where $\mathcal{F}=\mathcal{F}_{T}$ and $T>0$ is a fixed time horizon. Let $\BE$ be the expectation with respect to (w.r.t.) $\BP$.
In this space, we define a standard one-dimensional Brownian motion $W(t)$, $t\in[0, T]$.
We assume $\mathcal{F}_{t}=\sigma\{ W(s): 0\leqslant s\leqslant t\}\bigvee \mathcal{N}$, where
$\mathcal{N}$ is the totality of all the $\BP$-null sets of $\mathcal{F}$.

As usual, we denote by $\BR^{n}$ the $n$-dimensional real-valued Euclidean space with the Euclidean norm $|\cdot|$,
%by $\BR^{n\times m}$ the set of all $(n\times m)$ real-valued matrices,
by $\BR^{n}_{+}$ the set of vectors in $\BR^{n}$ whose components are all positive numbers,
by $\BR$ the set of real numbers,
and by $\mathbb{N}$ the set of natural numbers.
% and $\BR_{\gg1}=\BR^{1}_{+}.$
% the set of all positive real numbers.
%Here and hereafter, the superscript $\top$ denotes the transpose of vectors or matrices.
%and by $\mathbb{S}^{n} $ the set of all $(n\times n)$ real-valued symmetric matrices.
%%$\mathbb{S}^{n}_{+}$ (resp., $\mathbb{S}^{n}_{-}$) the set of $(n\times n)$ positive (resp., negative) semidefinite real-valued matrices,
%%and $\hat{\mathbb{S}}^{n}_{+}$ (resp., $\hat{\mathbb{S}}^{n}_{-}$) the set of $(n \times n)$ positive (resp., negative)
%%definite real-valued matrices.
%We use $I_{n}$ to denote the $n$-dimensional identity matrix.
%For any $S\in \mathbb{S}^n$, $c\in \BR$,
%we write $S\geq cI_{n}$ if $ y^{\top} S y\geq c|y|^2$ for any $y\in \BR^{n}$,
%and define $S\leq cI_{n}$ similarly.
%We denote the positive and negative parts of a constant $x$ as $x^{+}=\max\{x,0\}$ and $x^{-}=\max\{-x,0\}$.
We use the following spaces throughout the paper:\bigskip\\
\begin{tabular}{rll}
$L^{\infty}_{\mathcal{F}_T}(\BR^{n}):$ &the set of all $\BR^{n}$-valued $\mathcal{F}_T$-measurable essentially bounded random variables;\smallskip\\

$L^2_{\mathcal{F}_T}(\BR^{n}):$& the set of all $\BR^{n}$-valued $\mathcal{F}_T$-measurable random variables $\xi$ such that $\BE\big[ |\xi|^2 \big]<\infty$;\smallskip\\

%$L^{\infty}_{\mathcal{F}}(0,T;\BR^{n}):$& the set of all $\BR^{n}$-valued $\mathcal{F}_t$-adapted essentially bounded processes defined\\
%& on $[0, T ]$;\smallskip\\
$L^{\infty}_{\mathcal{F}}(0,T;\BR^{n}):$& the set of all $\mathcal{F}_t$-adapted essentially bounded processes
$v:[0,T]\times \Omega\rightarrow \BR^{n}$; \smallskip\\

$L^{\infty}_{\mathcal{F}}(0,T;\BR_{\gg1}):$& the set of all $\mathcal{F}_t$-adapted processes
$v:[0,T]\times \Omega\rightarrow (0,+\infty)$ such that there is a \smallskip\\
&constant $c>0$, $c^{-1}\leq v(t)\leq c$, a.e. $\dtp$; \smallskip\\

$L^2_{\mathcal{F}}(0,T;\BR^{n}):$& the set of all $\mathcal{F}_t$-adapted processes
$v:[0,T]\times \Omega\rightarrow \BR^{n}$ such that \smallskip\\
&$\BE\big[\int_0^T |v(t)|^2 \dt \big]<\infty$; \smallskip\\

$L^{2,\text{loc}}_{\mathcal{F}}(0,T;\BR^{n}):$& the set of all $\mathcal{F}_t$-adapted processes
$v:[0,T]\times \Omega\rightarrow \BR^{n}$ such that \smallskip\\
&$\BP\big(\int_0^T |v(t)|^2 \dt <\infty \big)=1$; \smallskip\\

$L^{2}_{\mathcal{F}}(C(0,T);\BR^{n}):$ & the set of all $\mathcal{F}_t$-adapted processes with continuous sample paths $v:[0,T]\times \Omega\rightarrow \BR^{n}$ \smallskip\\
& such that $\BE\big[\sup \limits_{t\in[0,T]} |v(t)|^2 \big]<\infty$;\smallskip\\

$L^{0}_{\mathcal{F}}(C(0,T);\BR^{n}):$ & the set of all $\mathcal{F}_t$-adapted processes with continuous sample paths.\smallskip\\

\end{tabular}
%These definitions are generalized in the obvious way to the cases that $\BR^{n}$ is replaced by $\BR$ or $\BR_{\gg1}$.

In our analysis, the equations or inequalities for stochastic processes (resp. random variables) hold generally in the sense that $\dtp$ almost everywhere (resp. $\dd\BP$ almost surely).
In addition, for notation simplicity, some arguments such as $s$, $t$, $\omega$ may be suppressed in some circumstances when no confusion occurs. The processes considered in this paper, except specified, are all stochastic, so we do not write the argument $\omega$ explicitly when defining or using them.

We recall the definition of BMO martingale, which is a short form of the martingale of bounded mean oscillation.
For any process $f\in L^{2}_{\mathcal{F}}(0,T;\BR)$, the process $\int_0^{\cdot} f(s)\dw(s)$ is a BMO martingale on $[0,T]$ if and only if there exists a positive constant $c$ such that
$$ \BE\Big[ \int_{\tau}^T |f(s)|^2\ds\;\;\Big| \; \;\mathcal{F}_{\tau} \Big]\leq c$$
hold for all $\{\mathcal{F}_t\}_{t\geq0}$-stopping times $\tau \leq T$.
From now on, we use $c$ to represent a generic positive constant, which can be different from line to line.
We set
$$L^{2,\: \mathrm{BMO}}_{\mathcal{F}}(0,T;\BR)=\Big\{f \in L^{2}_{\mathcal{F}}(0,T;\BR)\;\Big| \;
\int_0^{\cdot} f(s) \dw(s) \text{ is a BMO martingale on } [0,T]\Big\}.$$
%For $\Lambda\in L^{2,\: \mathrm{BMO}}_{\mathcal{F}}(0,T;\BR)$, we set
%$$\Big\|\int_0^{\cdot} \Lambda(s) \dw(s)\Big\|_{\mathrm{BMO}_2}\triangleq
%\sup\limits_{\tau\leq T}\Big( \operatorname*{ess\,sup} \BE\Big[ \int_{\tau}^T |\Lambda(s)|^2\ds\;\;\Big| \; \;\mathcal{F}_{\tau} \Big] \Big)^{\frac{1}{2}} <\infty,$$
%here and hereafter the $\sup \limits_{\tau\leq T}$ is taken over all $\{\mathcal{F}_t\}_{t\geq0}$-stopping times $\tau\leq T$.
For any $Z\in L^{2,\: \mathrm{BMO}}_{\mathcal{F}}(0,T;\BR)$, the Dol$\acute{\text{e}}$ans-Dade stochastic exponential
$$\mathcal{E}\Big(\int_0^{\cdot} Z(s) \dw(s) \Big)\triangleq
\exp\Big\{ -\frac{1}{2}\int_0^{\cdot} |Z(s)|^2 \ds +\int_0^{\cdot} Z(s) \dw (s)\Big\}$$
is a uniformly integrable martingale on $[0,T]$. Moreover, under the probability measure $\widetilde{\BP}$ defined by
$$\frac{\dd\widetilde{\BP}}{\dd\BP}\bigg|_{\mathcal{F}_T}=\mathcal{E}\Big(\int_0^{T} Z(s) \dw(s) \Big),$$
the process $\widetilde{W}(\cdot)=W(\cdot)-\int_0^{\cdot} Z(s)\ds$ is a standard Brownian motion.
If $\int_0^{\cdot} \Lambda(s)\dw(s)$ is a BMO martingale under $\BP$, then
$\int_0^{\cdot} \Lambda(s)\dd\widetilde{W}(s)$ is also a BMO martingale under $\widetilde{\BP}$.
For more details about BMO martingales, interested readers can refer to \cite{Kazamaki}.

\section{Problem Formulation } \label{PF}

We introduce the market model and two kinds of relative performance (the CARA utility case and the CRRA utility case) in the first subsection.
For a more rigorous model, we give the definitions of admissible strategy sets and Nash equilibrium in the second subsection,
which makes our game problems well-defined.

\subsection{A non--Markovian financial market and agents' preferences}
Consider a financial market consisting of a risk-free asset (the money market instrument or bond)
whose price is $S_0$ and a risky security (the stock) whose price is $S_1$. These asset prices are driven by the following stochastic differential equations (SDEs):
\begin{equation*}
\left\{
\begin{aligned}
\dd S_0(t)&= r(t)S_0(t)\dt,\\
\dd S_1(t)&= S_1(t)\big[\mu(t)\dt+\sigma(t)\dw(t)\big],\\
S_0(0)&=s_0,\,\,S_1(0)=s_1,
\end{aligned}
\right.
\end{equation*}
where $r\in L^{\infty}_{\mathcal{F}}(0,T;\BR)$ is the interest rate process,
$\mu\in L^{\infty}_{\mathcal{F}}(0,T;\BR)$ and
$\sigma\in L^{\infty}_{\mathcal{F}}(0,T;\BR_{\gg1})$ are the appreciation rate process and volatility rate process of the risky security.
The initial prices $s_0$ and $s_1$ are positive constants.
In the following, we often omit the argument $t$ for the processes $r$, $\mu$ and $\sigma$ for notation simplicity, but it needs to be remembered that these market parameters are $\mathcal{F}_t$-adapted stochastic processes.
Therefore, it is a non-Markovian financial market.
Let $\rho\triangleq \frac{\mu-r} { \sigma }$ denote the Sharpe ratio process.
$\rho$ is also called the market price of risk or risk premium.
 Clearly, $\rho\in L^{\infty}_{\mathcal{F}}(0,T;\BR)$. If $\rho$ is identical to zero, then there is no incentive for agents to invest in the stock. Hence, we assume $\rho$ is not identical to zero in the rest of this paper.

\begin{remark}
For notation simplicity, this paper only considers one stock. There is no intrinsic difficulty in generalizing to multiple-stock case where the analysis is a more cumbersome notation.
\end{remark}

There are $n$ competing agents (investors or fund managers) in this non-Markovian financial market.
We use $\mathbf{A}^{(j)}$ to represent the $j$th agent.
For $1\leq j \leq n$, let $X^{(j)}$ be the wealth process of $\mathbf{A}^{(j)}$.
We denote the arithmetic mean and geometric mean of the agents' wealth processes as
$$\overline{X}\triangleq \frac{1}{n}\sum^n_{j=1} X^{(j)} \quad\text{ and }\quad
\widehat{X}\triangleq \bigg[ \prod_{j=1}^{n}X^{(j)} \bigg]^{\frac{1}{n}}.$$
%For $1\leq j\leq n$, $x^{(j)}\in\BR_{\gg1}$ is the initial wealth of $\mathbf{A}^{(j)}$,
% and $\pi^{(j)}\in \mathcal{C} \subseteq L^{{2,\: \mathrm{BMO}}}_{\mathcal{F}^W}(0,T;\BR)$ is an admissible strategy of $\mathbf{A}^{(j)}$.
% As is usually the case for exponential and power/logarithmic risk preferences, $\pi^{(j)}$ is taken to be the absolute wealth and the fraction of wealth invested in the stock, respectively. The values $\pi^{(j)}$ may be negative, indicating that the agent shorts the stock.
We assume all agents have a common investment horizon $[0,T]$ and aim to maximize their expected utility of above average wealth at $T$.
The utility functions are agent-specific functions of both his own wealth $X^{(j)}$ and
a ``competition component'' $\overline{X}$ or $\widehat{X}$, which depends on all agents' wealth.
We study two representative cases: the CARA utility case and the CRRA utility case, which cover the most popular exponential, power and logarithmic utilities.

We first consider the CARA utility case, where we assume all agents use some CARA utility functions.
For $ 1\leq j \leq n$, $\pi^{(j)}$ represents the amount invested in the stock, and $(X^{(j)}-\pi^{(j)})$ is the amount invested in the risk-free asset. Then agents' self-financing wealth processes satisfy a system of SDEs:
\begin{equation}\label{201}
\left\{
\begin{aligned}
\dd X^{(j)}(t)=\,&\big[r(t)X^{(j)}(t)+ (\mu(t)-r(t))\pi^{(j)}(t) \big]\dt
+\sigma(t)\pi^{(j)}(t) \dw(t),\\
X^{(j)}(0)=\,&x^{(j)}\in \BR_{+},\quad 1\leq j \leq n.
\end{aligned}
\right.
\end{equation}
Because the objectives of the agents will be correlated, their admissible strategies' are also correlated, and their precise definitions will be given in the next section.
The CARA utility function with the absolute risk tolerance constant $\delta>0$ is defined as
\begin{equation*}
u_1(x;\delta)=-\exp\Big\{-\frac{x}{\delta}\Big\},~x\in \BR.
\end{equation*}
For $ 1\leq j \leq n$, the aim of $\mathbf{A}^{(j)}$ is to maximize
\begin{equation}
\label{202}
J_1^{(j)}\big( \pi^{(j)},\pi^{(-j)}; x^{(j)},x^{(-j)} \big)
=\BE\big[ u_1\big(X^{(j)}(T)-\theta_j\overline{X}(T);\delta_j\big) \big],
\end{equation}
where $\pi^{(-j)}\triangleq(\pi^{(1)}, \ldots ,\pi^{(j-1)}, $ $\pi^{(j+1)},\cdots , \pi^{(n)})$ is an admissible strategy vector of all agents except $\mathbf{A}^{(j)}$,
$x^{(-j)}$ is an initial wealth vector of all agents except $\mathbf{A}^{(j)}$,
$\delta_j\in(0,+\infty)$ is the absolute risk tolerance constant of $\mathbf{A}^{(j)}$,
and $\theta_j\in[0,1]$ is the competition weight of $\mathbf{A}^{(j)}$.
For simplicity, we call this $n$-agent game the game \eqref{201}-\eqref{202}.

In the CRRA utility case, we assume all agents use some CRRA utility functions. For $ 1\leq j \leq n$, $\pi^{(j)}$ represents the proportion invested in the stock, and $(1-\pi^{(j)})$ is the proportion invested in the risk-free asset.
Then agents' self-financing wealth processes satisfy a system of SDEs:
\begin{equation}\label{203}
\left\{
\begin{aligned}
\dd X^{(j)}(t)=\,&\big[r(t)+ (\mu(t)-r(t))\pi^{(j)}(t) \big]X^{(j)}(t)\dt
+ \sigma(t) \pi^{(j)}(t)X^{(j)}(t) \dw(t),\\
X^{(j)}(0)=\,&x^{(j)}\in \BR_{+},\quad 1\leq j \leq n.
\end{aligned}
\right.
\end{equation}
The CRRA utility function with the relative risk tolerance constant $\delta>0$ is defined, for $x>0$, as
\begin{equation*}
u_2(x;\delta)=
\begin{cases}
\frac{x^{1-\frac{1}{\delta}}}{1-\frac{1}{\delta}}, &\mbox{if~~} \delta\in(0,1)\cup(1,+\infty),\\
\log x, &\mbox{if~~}\delta=1.
\end{cases}
\end{equation*}
For $ 1\leq j \leq n$, the aim of $\mathbf{A}^{(j)}$ is to maximize
\begin{equation}
\label{204}
J_2^{(j)}\big( \pi^{(j)},\pi^{(-j)} ; x^{(j)},x^{(-j)} \big)
=\BE\big[ u_2\big(X^{(j)}(T)\widehat{X}(T)^{-\theta_j};\delta_j \big) \big],
\end{equation}
where $\theta_j\in[0,1]$ have the same meaning as in \eqref{202},
while $\delta_j\in(0,+\infty)$ is the relative risk tolerance constant of $\mathbf{A}^{(j)}$.
For simplicity, we call this $n$-agent game the game \eqref{203}-\eqref{204}.

\begin{remark}
The terms in utility functions in \eqref{202} and \eqref{204} can be written as
\begin{equation*}
\begin{aligned}
X^{(j)}(T)-\theta_j\overline{X}(T)&= (1-\theta_j)X^{(j)}(T)+\theta_j \big[ X^{(j)}(T)-\overline{X}(T) \big],\\
X^{(j)}(T)\big[\widehat{X}(T)\big]^{-\theta_j}&= \big[X^{(j)}(T)\big]^{1-\theta_j}
\bigg[ \frac{X^{(j)}(T)}{\widehat{X}(T)} \bigg]^{\theta_j},
\end{aligned}
\end{equation*}
where $X^{(j)}(T)$ is the absolute terminal wealth of $\mathbf{A}^{(j)}$, the difference $\big[ X^{(j)}(T)-\overline{X}(T) \big]$ and the ratio $\Big[\frac{X^{(j)}(T)}{\widehat{X}(T)}\Big]$ measure the relative terminal wealth.
The larger $\theta_j$ is, the more weight $\mathbf{A}^{(j)}$ places on the relative wealth, which means she/he is more competitive.
When $\theta_j = 0$, $\mathbf{A}^{(j)}$ does not care about relative wealth and she/he is not at all competitive, and the above two game problems degenerate into classical Morton's problems.
So we call $\theta_j$ the competition weight.
\end{remark}

\subsection{Admissible strategy sets and Nash equilibria}

For the above two $n$-agent games, our goal is to find their Nash equilibria in spaces
$\seta_1$ and $\seta_2$ respectively.
We first introduce the definitions of two admissible strategy sets and Nash equilibrium.

For any $\pi^{(j)}\in L^{2,\text{loc}}_{\mathcal{F}}(0,T;\BR)$, linear SDEs \eqref{201} and \eqref{203} admit unique solutions
$X^{(j)}\in L^{0}_{\mathcal{F}}(C(0,T);\BR)$. In order to let the functionals \eqref{202} and \eqref{204} make sense and the game problems to be well-defined, we define the following two admissible strategy sets.

\begin{definition}\label{admissible set 1}
In the CARA utility case, the admissible strategy set,
a subset of $\big(L^{2,\text{loc}}_{\mathcal{F}}(0,T;\BR)\big)^n$, is denoted by $\seta_1$.
% which is the product space of $n$ $L^{2,\text{loc}}_{\mathcal{F}}(0,T;\BR)$,
It holds that $(\pi^{(1)}, \pi^{(2)},\ldots ,\pi^{(n)})\in\seta_1 $ if and only if
\begin{enumerate}[(i)]
\item for all $1\leq j\leq n$, SDE \eqref{201} has a unique strong solution $X^{(j)}\in L^{0}_{\mathcal{F}}(C(0,T);\BR)$;\\
\item for any sequence of nondecreasing $\{\mathcal{F}_t\}_{t\geq0}$-stopping times $\tau_k$, $k\in\mathbb{N},$ such that $\BP\big(\lim\limits_{k\rightarrow\infty}\tau_k=T\big)=1$, it always holds that, for all $1\leq j \leq n$,
the family $\Big\{ \exp\big\{ -\frac{\psi(T\wedge \tau_k)}{\delta_j} \big[ X^{(j)}(T\wedge \tau_k)
-\theta_j\overline{X}(T\wedge \tau_k)\big]+\varphi(T\wedge \tau_k)\big\} \Big\}_{k\in \mathbb{N}}$ is uniformly integrable,
where processes $\psi$ and $\varphi$ are unique solutions of BSDEs \eqref{3201} and \eqref{3202}, respectively.
\end{enumerate}
\end{definition}

\begin{definition}\label{admissible set 2}
In the CRRA utility case,
the admissible strategy set, a subset of
$\big(L^{2,\text{BMO}}_{\mathcal{F}}(0,T;\BR)\big)^n$, is denoted by $\seta_2$.
It holds that $(\pi^{(1)}, \pi^{(2)},\ldots ,\pi^{(n)})\in\seta_2$ if and only if
\begin{enumerate}[(i)]
\item for all $1\leq j\leq n$, SDE \eqref{203} has a unique strong solution $X^{(j)}\in L^{0}_{\mathcal{F}}(C(0,T);\BR)$;\\
\item
for all $1\leq j\leq n$, BSDE \eqref{4101} admits a solution $(P^{(j)}, \Lambda^{(j)})\in L_{\mathcal{F}}^{\infty}(0,T;\BR)\times L_{\mathcal{F}}^{2,\: \mathrm{BMO}}(0,T;\BR)$;
\item
for any sequence of nondecreasing $\{\mathcal{F}_t\}_{t\geq0}$-stopping times $\tau_k$, $k\in\mathbb{N},$ such that $\BP\big(\lim\limits_{k\rightarrow\infty}\tau_k=T\big)=1$, it always holds that, for all $1\leq j \leq n$,
the family $\Big\{ \big[X^{(j)}(T\wedge\tau_k)\big]^{\beta_j} \big[\widehat{X}^{(-j)}(T\wedge\tau_k)\big]^{\gamma_j} e^{P^{(j)}(T\wedge\tau_k)} \Big\}_{k\in \mathbb{N}}$ is uniformly integrable,
where the constants $\beta_j$, $\gamma_j$ are defined in Section \ref{CRRA}, and $(P^{(j)},\Lambda^{(j)})$ are given in (ii).
\end{enumerate}
\end{definition}
Both $\seta_1$ and $\seta_2$ are nonempty, because $(0,0, \cdots,0)$ is an element of $\seta_1$ and $ \seta_2$.

\begin{remark}
For any $(\pi^{(1)},\pi^{(2)}, \cdots,\pi^{(n)})\in \seta_1$ (resp. $\in \seta_2$), the functional \eqref{202} (resp. \eqref{204}) makes sense (see the proofs of Theorem \ref{best response CARA}, Theorem \ref{best response 3} and Theorem \ref{best response 4}). Furthermore, the value functions defined in \eqref{3200} and \eqref{4200} satisfy
$$-\infty <V_1^{(j)}(x^{(j)},x^{(-j)};\pi^{(-j)})\leq 0,\quad
-\infty <V_2^{(j)}(x^{(j)},x^{(-j)};\pi^{(-j)})< +\infty,\quad 1\leq j\leq n.$$
Therefore, the above two game problems are well-defined.
\end{remark}

\begin{definition}
An admissible strategy $(\pi^{(1,*)},\pi^{(2,*)}, \cdots,\pi^{(n,*)})\in \seta_1$ (resp. $\in \seta_2$) is a Nash equilibrium for game \eqref{201}-\eqref{202} (resp. \eqref{203}-\eqref{204}), if for all $1\leq j\leq n$, and for all $\pi^{(j)}$ such that $(\pi^{(j)}, \pi^{(-j,*)})\in \seta_1$ (resp. $\in \seta_2$), one has
\begin{equation*}
\begin{aligned}
J_1^{(j)}(\pi^{(j)}, \pi^{(-j,*)};x^{(j)},x^{(-j)})
\leq J_1^{(j)}(\pi^{(j,*)}, \pi^{(-j,*)};x^{(j)},x^{(-j)})
\end{aligned}
\end{equation*}
%\mbox{\red{$\mathcal{C}_1^j$ is not well defined. we shall need $(\pi^{(j)}, \pi^{(-j,*)})\in C^{1}$ }}
\begin{equation*}
\begin{aligned}
\big( \text{resp. } J_2^{(j)}(\pi^{(j)}, \pi^{(-j,*)};x^{(j)},x^{(-j)})
\leq J_2^{(j)}(\pi^{(j,*)}, \pi^{(-j,*)};x^{(j)},x^{(-j)})\big),
\end{aligned}
\end{equation*}
where
$\pi^{(-j,*)}\triangleq(\pi^{(1,*)}, \ldots ,\pi^{(j-1,*)}, $ $\pi^{(j+1,*)},\cdots , \pi^{(n,*)})$.
\end{definition}

\subsection{ABR and BR methods}

In general, there are two steps to obtain a Nash equilibrium. The first step is to solve an optimal control problem, and the second step is to solve the $n$-variable system of equations.
We now introduce two methods to obtain a Nash equilibrium, both of which involve the above two two steps.

$\bullet$ \textbf{ABR Method}

\textbf{Step 1. (Finding an admissible best response of an agent)} For each fixed $j\in\{1,\cdots, n\}$, $k\in\{1,2\}$, and an arbitrary but fixed competitors'
strategy vector $\pi^{(-j)}$, a response strategy $\check{\pi}^{(j)}(\pi^{(-j)})$ is called an \textit{admissible best response} of $\mathbf{A}^{(j)}$, if
$$( \check{\pi}^{(j)}(\pi^{(-j)}),\pi^{(-j)})\in \seta_k$$
and for all $(x^{(j)},x^{(-j)})\in\BR^n_{+}$,
\begin{equation*}
\begin{aligned}
\sup \limits_{\pi^{(j)}\in \seta_k} J_k^{(j)}(\pi^{(j)}, \pi^{(-j)};x^{(j)},x^{(-j)})
= J_k^{(j)}(\check{\pi}^{(j)}(\pi^{(-j)}), \pi^{(-j)};x^{(j)},x^{(-j)})
\end{aligned}
\end{equation*}
where $\sup \limits_{\pi^{(j)}\in \seta_k}$ is taken over all $\pi^{(j)}$ such that $(\pi^{(j)},\pi^{(-j)})\in \seta_k$ here and hereafter.

\textbf{Step 2. (The strategies of all agents are mutually best responses)}
For all $1 \leq j \leq n$, let $\check{\pi}^{(j)}(\pi^{(-j)})$ be an admissible best response strategy of $\mathbf{A}^{(j)}$ found in Step 1
when $k=1$ (resp. $k=2$). Consider the following $n$-variable system of equations:
\begin{equation}\label{205}
\left\{
\begin{aligned}
\pi^{(1)}&=\check{\pi}^{(1)}(\pi^{(-1)}),\\
\pi^{(2)}&=\check{\pi}^{(2)}(\pi^{(-2)}),\\
&\vdots\\
\pi^{(n)}&=\check{\pi}^{(n)}(\pi^{(-n)}).\\
\end{aligned}
\right.
\end{equation}
If it has a solution $(\pi^{(1)},\pi^{(2)},\cdots,\pi^{(n)})$, then this solution is a Nash equilibrium of game \eqref{201}-\eqref{202} (resp. \eqref{203}-\eqref{204}).

The advantage of ABR method is that we can use it to obtain the existence and uniqueness of the Nash equilibrium.
%\begin{assumption}\label{ass1}
%For all $1\leq j\leq n$, $(x^{(j)},x^{(-j)})\in\BR^n_{+} $ and an arbitrary but fixed competitors'
%strategy vector $\pi^{(-j)}$, for any two strategies $\pi^{(j)}$, $\tilde{\pi}^{(j)}$
%such that $(\pi^{(j)},\pi^{(-j)})\in \seta_k$ and $(\tilde{\pi}^{(j)},\pi^{(-j)})\in \seta_k$, the following fact holds:
%$$J_k^{(j)}(\pi^{(j)}, \pi^{(-j)};x^{(j)},x^{(-j)})= J_k^{(j)}(\tilde{\pi}^{(j)}, \pi^{(-j)};x^{(j)},x^{(-j)}),\text{ if and only if }
%\pi^{(j)}=\tilde{\pi}^{(j)},$$
%where $k\in \{1,2\}$.
%\end{assumption}

\begin{theorem} \label{the existence and uniqueness of NE}
%Let Assumption \ref{ass1} hold.
Assume that for all agents, their admissible best responses are unique when $k=1$ (resp. $k=2$).
Then, $(\pi^{(1)},\pi^{(2)},\cdots,\pi^{(n)})$ is a Nash equilibrium of game \eqref{201}-\eqref{202} (resp. \eqref{203}-\eqref{204}) if and only if $(\pi^{(1)},\pi^{(2)},\cdots,\pi^{(n)})$ is a solution of $n$-variable system of equations \eqref{205}.
%\textbf{(The existence of a Nash equilibrium)}
% If for an admissible best responses, \eqref{205} has a solution, then we can conclude that there exsits Nash equilibrium of game \eqref{201}-\eqref{202} (resp. \eqref{203}-\eqref{204});
% If for all admissible best responses, \eqref{205} has no solution, then we can conclude that there is no Nash equilibrium of game \eqref{201}-\eqref{202} (resp. \eqref{203}-\eqref{204});
%
% \textbf{(The uniqueness of a Nash equilibrium)}
%If Assumption \ref{ass1} holds for $k=1$ (resp. $k=2$) and for all agents, their admissible best responses are unique and \eqref{205} has a unique solution, then we can conclude that there is a unique Nash equilibrium of game \eqref{201}-\eqref{202} (resp.\eqref{203}-\eqref{204}).
\end{theorem}

However, in general, it is difficult to prove the admissibility of a response strategy. The following method may be more practical.

$\bullet$ \textbf{BR Method}

\textbf{Step 1. (Finding a best response of an agent)}
For each fixed $j\in\{1,\cdots, n\}$, $k\in\{1,2\}$,
and an arbitrary but fixed competitors' strategy vector $\pi^{(-j)}$,
a response strategy $\check{\pi}^{(j)}(\pi^{(-j)})$ is called a \textit{best response} of $\mathbf{A}^{(j)}$, if
for all $(x^{(j)},x^{(-j)})\in\BR^n_{+}$, $J_k^{(j)}(\check{\pi}^{(j)}(\pi^{(-j)}), \pi^{(-j)};x^{(j)},x^{(-j)})$ is well-defined and
\begin{equation*}
\begin{aligned}
\sup \limits_{\pi^{(j)}\in \seta_k} J_k^{(j)}(\pi^{(j)}, \pi^{(-j)};x^{(j)},x^{(-j)})
\leq J_k^{(j)}(\check{\pi}^{(j)}(\pi^{(-j)}), \pi^{(-j)};x^{(j)},x^{(-j)}).
\end{aligned}
\end{equation*}

\textbf{Step 2. (The strategies of all agents are mutually best responses)}
For all $1 \leq j \leq n$, let $\check{\pi}^{(j)}(\pi^{(-j)})$ be a best response strategy of $\mathbf{A}^{(j)}$ found in Step 1
when $k=1$ (resp. $k=2$).
If $n$-variable system of equations \eqref{205} has a solution $(\pi^{(1)},\pi^{(2)},\cdots,\pi^{(n)})$ in $\seta_1$
(resp. $\seta_2$), then this solution is a Nash equilibrium of game \eqref{201}-\eqref{202} (resp. \eqref{203}-\eqref{204}).

For BR method, note that in Step 1, we do not need to prove the best response strategy is admissible.
We just prove the Nash equilibrium obtained in Step 2 is admissible.
In fact, in order to obtain the admissibility of a best response strategy of $\mathbf{A}^{(j)}$, we need to prove that for an arbitrary choice $\pi^{(-j)}$, $(\check{\pi}^{(j)}(\pi^{(-j)}),\pi^{(-j)})$ is admissible. While, for the admissibility of the Nash equilibrium, we only need to prove that $(\check{\pi}^{(j)}(\pi^{(-j)}),\pi^{(-j)})$ is admissible with $\pi^{(-j)}=\pi^{(-j,*)}$ and $(\pi^{(j,*)},\pi^{(-j,*)})$ is a given equilibrium strategy.
Obviously, the latter is less difficult.

In the following sections, we use the ABR method to discuss the Nash equilibrium for the CARA utility case, and we use the BR method to give a Nash equilibrium for the CRRA utility case.

\section{The CARA utility case}\label{CARA}

In this section, we find the Nash equilibrium for game \eqref{201}-\eqref{202} in $\seta_1$ by the ABR method.
We first resolve the solvability of two BSDEs, one of which is a quadratic BSDE and the other is a linear BSDE with unbounded coefficients,
in space $L_{\mathcal{F}}^{\infty}(0,T;\BR)\times L_{\mathcal{F}}^{2,\: \mathrm{BMO}}(0,T;\BR)$ in Theorem \ref{Solvability of BSDE 1}.
Then, in Theorem \ref{best response CARA}, for an arbitrary but fixed competitors' strategy vector $\pi^{(-j)}$, we obtain a unique admissible best response of $\mathbf{A}^{(j)}$, which is a wealth-dependent strategy.
Lastly, solving \eqref{205}, we show three cases for Nash equilibrium in Theorem \ref{NE 1}.

Let $\overline{X}^{(-j)}\triangleq \frac{1}{n} \sum_{k=1,k\neq j}^{n}X^{(k)}$.
Then from \eqref{201}, we have
\begin{equation}\label{301}
\left\{
\begin{aligned}
\dd \overline{X}^{(-j)}(t)
=\,&\Big[ r \overline{X}^{(-j)}(t) +\frac{\mu-r}{n}\sum_{k\neq j} \pi^{(k)}(t) \Big] \dd t
+\frac{\sigma}{n}\sum_{k\neq j} \pi^{(k)}(t)\dd W(t) ,\\
\overline{X}^{(-j)}(0) =\,& \bar{x}^{(-j)}\triangleq \frac{1}{n} \sum_{k\neq j} x^{(k)},\quad 1\leq j \leq n.
\end{aligned}
\right.
\end{equation}
And \eqref{202} can be written as
\begin{equation*}
J_1^{(j)}\big( \pi^{(j)},\pi^{(-j)} ; x^{(j)},x^{(-j)} \big)=
-\BE\Big[ \exp\Big\{ -\frac{1}{\delta_j}\Big[ \Big( 1-\frac{\theta_j}{n}\Big)X^{(j)}(T)
-\theta_j\overline{X}^{(-j)}(T) \Big] \Big\} \Big].
\end{equation*}
For an arbitrary but fixed competitors' strategy vector $\pi^{(-j)}$, the value function for $\mathbf{A}^{(j)}$ in the CARA utility case is defined by
\begin{equation}\label{3200}
V_1^{(j)}(x^{(j)},x^{(-j)};\pi^{(-j)})\triangleq \sup\limits_{\pi^{(j)}\in \seta_1}
J_1^{(j)}\big( \pi^{(j)},\pi^{(-j)} ; x^{(j)},x^{(-j)} \big),\quad \forall\, (x^{(j)},x^{(-j)})\in\BR^n_{+}.
\end{equation}

For $(t,\psi,\eta, \Delta )\in [0,T]\times\BR_{+} \times \BR \times \BR$, we set
\begin{equation*}
\begin{aligned}
f_1(t,\psi, \eta)\triangleq -r(t)\psi+\rho(t)\eta+\frac{\eta^2}{\psi},\quad
f_2(t,\psi, \eta, \Delta )\triangleq \frac{1}{2} \Big( \rho(t)+\frac{\eta}{\psi} \Big)^2+\Big( \rho(t)+\frac{\eta}{\psi} \Big)\Delta.
\end{aligned}
\end{equation*}
In the following, we often omit the argument $t$ for $f_1$ and $f_2$ without causing confusion.
Let us introduce the following two $1$-dimensional BSDEs:
\begin{equation}
\label{3201}
\left\{
\begin{aligned}
\dd \psi(t)=\,&f_1(\psi(t),\eta(t))\dt +\eta(t)\dw(t),\\
\psi(T)=\,&1, \,\, \psi>0,
\end{aligned}
\right.
\end{equation}
and
\begin{equation}
\label{3202}
\left\{
\begin{aligned}
\dd \varphi(t)=\,&f_2(\psi(t),\eta(t),\Delta(t))\dt +\Delta(t)\dw(t),\\
\varphi(T)=\,&0.
\end{aligned}
\right.
\end{equation}

Although BSDEs \eqref{3201}-\eqref{3202} form a coupled BSDE system, we can treat them separately since \eqref{3201} does not depend on \eqref{3202}.
Clearly \eqref{3201} is a $1$-dimensional quadratic BSDE on $(\psi,\eta)$.
Let $(\psi,\eta)$ be a solution of \eqref{3201}, then \eqref{3202} is a $1$-dimensional linear BSDE on $(\varphi, \Delta)$.
Because some coefficients in \eqref{3202} depend on $\eta$, they are unbounded.
In order to make some stochastic integrals have good integrability, we consider solutions of BSDEs in
$L_{\mathcal{F}}^{\infty}(0,T;\BR)\times L_{\mathcal{F}}^{2,\: \mathrm{BMO}}(0,T;\BR)$ in this paper.

\begin{definition}
A pair $(\psi,\eta)$ is called a solution to BSDE \eqref{3201} (resp. \eqref{3202}), if it satisfies \eqref{3201} (resp. \eqref{3202}), and
$(\psi,\eta)\in L_{\mathcal{F}}^{\infty}(0,T;\BR_{\gg1})\times L_{\mathcal{F}}^{2,\: \mathrm{BMO}}(0,T;\BR)$
(resp. $(\varphi,\Delta)\in L_{\mathcal{F}}^{\infty}(0,T;\BR)\times L_{\mathcal{F}}^{2,\: \mathrm{BMO}}(0,T;\BR)$).
\end{definition}
The solutions for the subsequent BSDEs are defined similarly and will not be repeated.

\begin{theorem} \label{Solvability of BSDE 1}
BSDEs \eqref{3201} and \eqref{3202} admit unique solutions in
$L_{\mathcal{F}}^{\infty}(0,T;\BR)\times L_{\mathcal{F}}^{2,\: \mathrm{BMO}}(0,T;\BR)$.
\end{theorem}
\begin{proof}
Let $(\widetilde{\psi},\widetilde{\eta}\,)=(\ln \psi, \frac{\eta}{\psi})$. Then $(\widetilde{\psi},\widetilde{\eta}\,)$ satisfies the following $1$-dimensional quadratic BSDE:
\begin{equation}\label{3207}
\left\{
\begin{aligned}
\dd \widetilde{\psi}(t)=\,&-\Big[ r-\rho\widetilde{\eta}-\frac{1}{2}{\widetilde{\eta}\,}^2 \Big]\dt +\widetilde{\eta}(t)\dw(t),\\
\widetilde{\psi}(T)=\,&0.
\end{aligned}
\right.
\end{equation}
From Theorem 2.4 in \cite{Fan Hu Tang}, BSDE \eqref{3207} admits a unique solution in
$L_{\mathcal{F}}^{\infty}(0,T;\BR)\times L_{\mathcal{F}}^{2,\: \mathrm{BMO}}(0,T;\BR)$. Thus, BSDE \eqref{3201} admits a unique solution
in $L_{\mathcal{F}}^{\infty}(0,T;\BR_{\gg1})\times L_{\mathcal{F}}^{2,\: \mathrm{BMO}}(0,T;\BR)$.
Let $(\psi,\eta)$ be the unique solution of \eqref{3201},
from Lemma 3.5 in \cite{Xu arXiv1}, BSDE \eqref{3202} admits a unique solution in
$L_{\mathcal{F}}^{\infty}(0,T;\BR)\times L_{\mathcal{F}}^{2,\: \mathrm{BMO}}(0,T;\BR)$.
\end{proof}

\begin{theorem}[The unique admissible best response of $\mathbf{A}^{(j)}$] \label{best response CARA}
In the CARA utility case, for all $1\leq j\leq n$, for an arbitrary but fixed competitors' strategy vector $\pi^{(-j)}$, the unique admissible best response of $\mathbf{A}^{(j)}$ is
\begin{equation}\label{3203}
\begin{aligned}
\check{\pi}^{(j)}(t,X^{(j)},\pi^{(-j)}) =\,& -\frac{n\eta(t)} {(n-\theta_j)\sigma(t)\psi(t)}\big[ X^{(j)}(t)-\theta_j\overline{X}(t)\big]
+\frac{\theta_j } { n-\theta_j } \sum_{k\neq j} \pi^{(k)}(t)\\
&+\frac{n\delta_j} {(n-\theta_j) \sigma(t)\psi(t)}
\Big[ \rho(t) + \Delta(t) +\frac{\eta(t)}{ \psi(t)} \Big],
\end{aligned}
\end{equation}
and the value function is
\begin{equation}\label{3204}
V_1^{(j)}(x^{(j)},x^{(-j)};\pi^{(-j)})= -\exp\Big\{ -\frac{\psi(0) }{\delta_j}
\big( x^{(j)}-\theta_j\bar{x}\big)+\varphi(0)\Big\},
\end{equation}
where $\bar{x}\triangleq\frac{1}{n}\sum_{j=1}^n x^{(j)}$, $(\psi,\eta)$ and $(\varphi,\Delta)$ are unique solutions of \eqref{3201} and \eqref{3202}, respectively.
\end{theorem}
\begin{proof}
%\red{We may try to express every thing in terms of $(\widetilde{\psi},\widetilde{\eta}\,)$ since $\frac{\eta(t)}{ \psi(t)}$ appears in $\pi$. We need its integrability.}
%\blue{We only need $\frac{\eta(t)}{ \psi(t)}\in L^{2,\mathrm{BMO}}_{\mathcal{F}}(0,T;\BR) $ in this proof. This is clear, because $(\psi,\eta)\in L_{\mathcal{F}}^{\infty}(0,T;\BR_{\gg1})\times L_{\mathcal{F}}^{2,\: \mathrm{BMO}}(0,T;\BR)$.}
%
We first prove the response strategy given in \eqref{3203} is admissible.
Because the processes $X^{(j)}$ and $\overline{X}$ are continuous, they are almost surely bounded on $[0,T]$, and
$(\psi,\eta)\in L_{\mathcal{F}}^{\infty}(0,T;\BR_{\gg1})\times L_{\mathcal{F}}^{2,\: \mathrm{BMO}}(0,T;\BR)$, $(\varphi,\Delta)\in L_{\mathcal{F}}^{\infty}(0,T;\BR)\times L_{\mathcal{F}}^{2,\: \mathrm{BMO}}(0,T;\BR)$.
It is obvious $\check{\pi}^{(j)}(X^{(j)},\pi^{(-j)})\in L^{2,\text{loc}}_{\mathcal{F}}(0,T;\BR)$.

SDE \eqref{201} with the strategy $\pi^{(j)}=\check{\pi}^{(j)}(X^{(j)},\pi^{(-j)})$ is:
\begin{equation*}
\left\{
\begin{aligned}
\dd X^{(j)}(t)=\,&\Big[ \Big(r-\frac{\rho\eta}{\psi} \Big) X^{(j)}+ \frac{\theta_j\rho\eta} {(n-\theta_j)\psi} \sum_{k\neq j} X^{(k)}
+\frac{\theta_j(\mu-r) } { n-\theta_j } \sum_{k\neq j} \pi^{(k)}
+\frac{n\delta_j\rho} {(n-\theta_j) \psi}
\Big( \rho + \Delta +\frac{\eta}{ \psi} \Big) \Big]\dt\\
&+\Big[ -\frac{\eta}{\psi} X^{(j)} +\frac{\theta_j\eta}{(n-\theta_j)\psi}\sum_{k\neq j} X^{(k)}
+\frac{\theta_j\sigma } { n-\theta_j } \sum_{k\neq j} \pi^{(k)}
+\frac{n\delta_j} {(n-\theta_j)\psi}
\Big( \rho+\Delta +\frac{\eta}{ \psi} \Big)\Big] \dw(t),\\
X^{(j)}(0)=\,&x^{(j)}\in \BR_{+}.
\end{aligned}
\right.
\end{equation*}
For all $1\leq j\leq n$, above SDE is a linear SDE, therefore it has a unique strong solution in $L^{0}_{\mathcal{F}}(C(0,T);\BR)$.

For an arbitrary but fixed competitors' strategy vector $\pi^{(-j)}$, applying It\^{o}'s formula to
$\exp\big\{ \frac{\psi}{\delta_j}\big[ \big( 1-\frac{\theta_j}{n}\big) \check{X}^{(j)}-\theta_j\overline{X}^{(-j)}\big]+\varphi\big\}$,
where $\check{X}^{(j)}$ is the solution of \eqref{201} with $\pi^{(j)}=\check{\pi}^{(j)}(X^{(j)},\pi^{(-j)})$ and $\overline{X}^{(-j)}$ is the solution of \eqref{301}, we get
\begin{equation*}
\frac{\dd \,\exp\big\{ -\frac{\psi}{\delta_j} \big[ \big( 1-\frac{\theta_j}{n} \big)\check{X}^{(j)}
-\theta_j\overline{X}^{(-j)}\big]+\varphi \big\} }
{ \exp\big\{ -\frac{\psi}{\delta_j} \big[ \big( 1-\frac{\theta_j}{n} \big)\check{X}^{(j)}
-\theta_j\overline{X}^{(-j)}\big]+\varphi \big\} }
= -\Big( \rho+\frac{\eta}{\psi}\Big) \dw(t).
\end{equation*}
Thus, we have
$$\exp\Big\{ -\frac{\psi(t)}{\delta_j}\Big[ \Big( 1-\frac{\theta_j}{n} \Big)\check{X}^{(j)}(t)
-\theta_j\overline{X}^{(-j)}(t) \Big]+\varphi(t)\Big\}
=e^{ -\frac{\psi(0) }{\delta_j} ( x^{(j)}-\theta_j\bar{x} )+\varphi(0) }
\mathcal{E}\Big(\int_0^{t} -\Big(\rho+\frac{\eta}{\psi} \Big) \dw(s) \Big).$$
Because $-\big( \rho+\frac{\eta}{\psi}\big) \in L^{2,\: \mathrm{BMO}}_{\mathcal{F}}(0,T;\BR)$, we see that $\exp\big\{ -\frac{\psi(t)}{\delta_j}\big[ \big( 1-\frac{\theta_j}{n} \big)\check{X}^{(j)}(t) -\theta_j\overline{X}^{(-j)} (t)\big]+\varphi(t)\big\}$ is a uniformly integrable martingale on $[0,T]$. Hence, we can conclude that, for an arbitrary but fixed competitors' strategy vector $\pi^{(-j)}$,
$( \check{\pi}^{(j)}(X^{(j)},\pi^{(-j)}), \pi^{(-j)} )\in \seta_1$.

Next, we prove the optimality of the response strategy.
Because in game \eqref{201}-\eqref{202}, all agents are symmetric. We just need to prove that for each fixed $j\in\{1,\cdots,n\}$,
\eqref{3203} is the unique admissible best response of $\mathbf{A}^{(j)}$.
For an arbitrary but fixed competitors' strategy vector $\pi^{(-j)}$, and any $\pi^{(j)}$ such that $(\pi^{(j)},\pi^{(-j)})\in \seta_1$, applying It\^{o}'s formula to $ -\exp\big\{ -\frac{\psi}{\delta_j}\big[ \big( 1-\frac{\theta_j}{n} \big) X^{(j)}-\theta_j\overline{X}^{(-j)}\big]+\varphi\big\}$, where $X^{(j)}$ is the solution of \eqref{201} and $\overline{X}^{(-j)}$ is the solution of \eqref{301}, we get
\begin{equation*}
\begin{aligned}
&-\exp\Big\{ -\frac{\psi(t)}{\delta_j} \Big[ \Big( 1-\frac{\theta_j}{n} \Big)X^{(j)}(t)
-\theta_j\overline{X}^{(-j)}(t)\Big]+\varphi(t)\Big\}
+\exp\Big\{-\frac{\psi(0)}{\delta_j} \big( x^{(j)}-\theta_j\bar{x} \big)+\varphi(0)\Big\} \\
=\,&-\frac{1}{2}\int_0^t e^{ -\frac{\psi}{\delta_j} \big[ ( 1-\frac{\theta_j}{n} )X^{(j)}-\theta_j\overline{X}^{(-j)}\big] +\varphi}
\Big(1-\frac{\theta_j}{n}\Big)^2\frac{\sigma^2\psi^2}{\delta_j^2}
\Big( \pi^{(j)}-\check{\pi}^{(j)}(X^{(j)},\pi^{(-j)}) \Big)^2\ds\\
&-\int_0^t e^{ -\frac{\psi}{\delta_j} \big[ ( 1-\frac{\theta_j}{n} )X^{(j)}-\theta_j\overline{X}^{(-j)}\big]+\varphi }
\Big( \Delta-\frac{\sigma\psi}{\delta_j}
\big[ \big(1-\frac{\theta_j}{n}\big)\pi^{(j)}-\frac{\theta_j}{n}\sum_{k\neq j} \pi^{(k)} \big]
-\frac{ \eta}{\delta_j}\big[ X^{(j)}-\theta_j\overline{X}\, \big]\Big) \dw(s),
\end{aligned}
\end{equation*}
where $\check{\pi}^{(j)} $ is given by \eqref{3203}.
Noting that the processes $X^{(j)}$ and $\overline{X}^{(-j)}$ are continuous, so almost surely bounded on $[0, T]$. Therefore, the stochastic integral in the above equation is a local martingale. Hence, there exists an nondecreasing sequence of $\{\mathcal{F}_t\}_{t\geq0}$-stopping times $\tau_k$, $k\in\mathbb{N}$, satisfying $\tau_k\rightarrow T$ as $k\rightarrow \infty$ such that, for all $k$,
\begin{equation*}
\begin{aligned}
&-\BE\Big[\exp\Big\{ -\frac{\psi(T\wedge \tau_k)}{\delta_j} \Big[ \Big( 1-\frac{\theta_j}{n} \Big)X^{(j)}(T\wedge \tau_k)
-\theta_j\overline{X}^{(-j)}(T\wedge \tau_k)\Big]+\varphi(T\wedge \tau_k)\Big\} \Big]\\
=\,&\BE\Big[-\frac{1}{2} \int_0^{T\wedge \tau_k} e^{ -\frac{\psi}{\delta_j} \big[(1-\frac{\theta_j}{n} )X^{(j)}-\theta_j\overline{X}^{(-j)}\big]+\varphi }
\Big(1-\frac{\theta_j}{n}\Big)^2\frac{\sigma^2\psi^2}{\delta_j^2}
\Big( \pi^{(j)}-\check{\pi}^{(j)}(X^{(j)},\overline{X}^{(-j)},\pi^{(-j)}) \Big)^2 \ds\Big]\\
&-\exp\Big\{-\frac{\psi(0)}{\delta_j} \big( x^{(j)}-\theta_j\bar{x} \big) + \varphi(0)\Big\}.
\end{aligned}
\end{equation*}
By Definition \ref{admissible set 1} and the equivalence between uniform integrability and $\mathcal{L}^1$ convergence,
the expectation on the left-hand side is convergent as $k \rightarrow \infty$.
Meanwhile, by the monotone convergence theorem, the right-hand side also converges as $k \rightarrow \infty$. Thus, we obtain
\begin{align*}
&\quad\; J_1^{(j)}\big( \pi^{(j)},\pi^{(-j)} ; x^{(j)},x^{(-j)}\big)\\
&=-\BE\Big[\exp\Big\{ -\frac{\psi(T)}{\delta_j} \Big[ \Big( 1-\frac{\theta_j}{n} \Big)X^{(j)}(T)
-\theta_j\overline{X}^{(-j)}(T)\Big]+\varphi(T)\Big\} \Big]\\
&=\BE\Big[-\frac{1}{2} \int_0^{T} e^{ -\frac{\psi}{\delta_j} \big[(1-\frac{\theta_j}{n} )X^{(j)}-\theta_j\overline{X}^{(-j)}\big]+\varphi }
\Big(1-\frac{\theta_j}{n}\Big)^2\frac{\sigma^2\psi^2}{\delta_j^2}
\Big( \pi^{(j)}-\check{\pi}^{(j)}(X^{(j)},\pi^{(-j)}) \Big)^2 \ds\Big]\\
&\qquad~~-\exp\Big\{-\frac{\psi(0)}{\delta_j} \big( x^{(j)}-\theta_j\bar{x} \big) + \varphi(0)\Big\}\\
&\leq -\exp\Big\{-\frac{\psi(0)}{\delta_j} \big( x^{(j)}-\theta_j\bar{x} \big) +\varphi(0)\Big\} ,
\end{align*}
and the equality holds if and only if $\pi^{(j)}=\check{\pi}^{(j)}(X^{(j)},\pi^{(-j)})$.
 We finish the proof.
\end{proof}

\begin{remark}
In Theorem \ref{best response CARA}, the unique admissible best response of $\mathbf{A}^{(j)}$ depends on the other agents' strategies $\pi^{(-j)}$, however, the value function of $\mathbf{A}^{(j)}$ has nothing to do with other agents' strategies $\pi^{(-j)}$. Thus, if an equilibrium exists, then its value function for $\mathbf{A}^{(j)}$ $(1\leq j\leq n)$ is given by \eqref{3204}.
\end{remark}

%\begin{assumption}\label{ass1}
%$\{(t):\rho(t)\neq 0\} $ is not a $\dtp$-null set.
%\end{assumption}

\begin{theorem}[Nash equilibrium in the CARA utility case] \label{NE 1}
Define two constants
$$\bar{\theta}\triangleq \frac{1}{n} \sum_{j=1}^n \theta_j\in[0,1],\quad
\bar{\delta}\triangleq\frac{1}{n}\sum_{j=1}^n \delta_j\in(0,+\infty).$$
Let $(\psi,\eta)$ and $(\varphi,\Delta)$ be the unique solutions of \eqref{3201} and \eqref{3202}, respectively. We then have
\begin{enumerate}[(i)]
\item When $\bar{\theta}<1$, there exists a unique Nash equilibrium in $\seta_1$, given by
\begin{equation}\label{3105}
\begin{aligned}
\pi^{(j,*)}(t) = -\frac{ \eta(t) }{\sigma(t)\psi(t)}X^{(j)}(t)
+\Big( \delta_j+ \frac{\bar{\delta}\theta_j}{1-\bar{\theta}} \Big)
\frac{ \rho(t)+\Delta(t)+\frac{\eta(t)}{\psi(t)} }{\sigma(t)\psi(t)},\,\, 1\leq j\leq n.
\end{aligned}
\end{equation}
%\item When $\bar{\theta}=1$ and $\rho+\Delta+\frac{\eta}{\psi} \neq 0$, there is no Nash equilibrium in $\seta_1$;
\item When $\bar{\theta}=1$ and
\begin{equation}\label{case3}
\int_{0}^{T} \BE\Big[r(s)+\frac{1}{2}|\rho(s)|^{2}\Big]\ds =\int_{0}^{T} \Big[r(s)+\frac{1}{2}|\rho(s)|^{2}\Big]\ds +\int_{0}^{T}\rho(s)\dw(s),
\end{equation} there exists infinitely many Nash equilibria in $\seta_1$, given by
$$\pi^{(j,*)}=\pi^{(n,*)}+\frac{\rho}{\sigma}\Big(X^{(j)}-X^{(n)}\Big), \quad
\forall\,\pi^{(n,*)}\in L^{2,\text{loc}}_{\mathcal{F}}(0,T;\BR), \quad 1\leq j\leq n-1.$$
\item When $\bar{\theta}=1$, but \eqref{case3} does not hold, there is no Nash equilibrium in $\seta_1$.
\end{enumerate}
\end{theorem}
\begin{proof}
From Theorem \ref{the existence and uniqueness of NE} and Theorem \ref{best response CARA}, to obtain a Nash equilibrium in the CARA utility case, we need to solve $n$-variable system of equations \eqref{205}, where $\check{\pi}^{(j)}(\pi^{(-j)} )$ is given by \eqref{3203}.
We set
$$\bar{\pi}\triangleq \frac{1}{n}\sum_{k=1}^n \pi^{(k)}
=\frac{1}{n}\pi^{(j)}+\frac{1}{n}\sum_{k\neq j} \pi^{(k)}.$$
Using the above expression, we can express \eqref{205} as the following $n$-variable system of equations
\begin{equation}\label{3205}
\begin{aligned}
\pi^{(j)}= -\frac{\eta}{\sigma\psi}\big[ X^{(j)}-\theta_j\overline{X}\,\big]+\theta_j\bar{\pi}+
\frac{\delta_j}{\sigma\psi} \Big( \rho +\Delta+\frac{\eta}{\psi} \Big), \quad 1\leq j\leq n.
\end{aligned}
\end{equation}
Then, averaging over $j=1,2, \ldots n$ yields
\begin{equation}\label{3206}
\begin{aligned}
\bar{\pi}= -\frac{\eta(1-\bar{\theta})}{\sigma\psi}\overline{X}+\bar{\theta}\bar{\pi}
+\frac{\bar{\delta}}{\sigma\psi}\Big( \rho +\Delta+\frac{\eta}{\psi} \Big) .
\end{aligned}
\end{equation}
%To make the equalities in \eqref{3205} hold, the equality \eqref{3206} must hold as well.

If $\bar{\theta}<1$, then \eqref{3206} implies $$\bar{\pi}= -\frac{\eta}{\sigma\psi}\overline{X}+\frac{\bar{\delta}}{(1-\bar{\theta})\sigma\psi } ( \rho+\Delta +\frac{\eta}{\psi} ).$$
Taking it back to \eqref{3205}, we obtain the Nash equilibrium \eqref{3105}.
The claim (i) is proved.

%From now on we study the case $\bar{\theta}=1$.
Before proving the claims (ii) and (iii), we first prove that \eqref{case3} holds if and only if there is $\psi$ satisfying
\begin{equation}\label{32002}
\begin{cases}
\dd \psi(t)=-r(t)\psi(t)\dt -\rho(t)\psi(t)\dw(t),\\
\psi(T)=1, \,\, \psi>0.
\end{cases}
\end{equation}
In fact, \eqref{32002} implies
\begin{equation*}
\log\psi(t)=\log\psi(0)-\int_{0}^{t} \Big[r(s)+\frac{1}{2}|\rho(s)|^{2}\Big]\ds -\int_{0}^{t}\rho(s)\dw(s).
\end{equation*}
Setting $t=T $ and taking expectation yield
\begin{equation*}
\psi(0)=\exp\Big(\int_{0}^{T} \BE\Big[r(s)+\frac{1}{2}|\rho(s)|^{2}\Big]\ds \Big).
\end{equation*}
Combing the above two leads to
\begin{equation*}
\log\psi(t)=\int_{0}^{T} \BE\Big[r(s)+\frac{1}{2}|\rho(s)|^{2}\Big]\ds -\int_{0}^{t} \Big[r(s)+\frac{1}{2}|\rho(s)|^{2}\Big]\ds -\int_{0}^{t}\rho(s)\dw(s).
\end{equation*}
Therefore, \eqref{32002} admits a solution if and only if \eqref{case3} holds.

We next prove that \eqref{case3} holds if and only if $\rho+\Delta+\frac{\eta}{\psi} = 0$.
\begin{itemize}
\item[$\Longrightarrow$] Suppose \eqref{case3} holds, then there is $\psi$ satisfying \eqref{32002}. Let $\eta=-\rho\psi$.
One can easily check that $(\psi,\eta)$ and $(\varphi,\Delta)\equiv(0,0)$ are the unique solutions to BSDEs \eqref{3201} and \eqref{3202}. It hence follows that $\rho+\Delta+\frac{\eta}{\psi} = 0$.
\item[$\Longleftarrow$] Suppose $\rho+\Delta+\frac{\eta}{\psi} = 0$. Taking it into the definition, we have $f_2(\psi,\eta,\Delta)=-\frac{1}{2}\Delta^2$. Consequently, the unique solution of \eqref{3202} is $(\varphi,\Delta)\equiv(0,0)$, which leads to $ \rho+\frac{\eta}{\psi} =\rho+\Delta+\frac{\eta}{\psi} = 0$. Then, BSDE \eqref{3201} becomes \eqref{32002}.
Thus, \eqref{32002} has a solution, so \eqref{case3} holds.
\end{itemize}
%
%Now suppose \eqref{case3} holds and $\psi$ satisfies \eqref{32002}. Set $\eta=-\rho\psi$.
%One can easily check that $(\psi,\eta)$ and $(\varphi,\Delta)\equiv(0,0)$ are the unique solutions to BSDEs \eqref{3201} and \eqref{3202}.
%This implies $\rho+\Delta+\frac{\eta}{\psi} = 0$ if \eqref{case3} holds.

Suppose $\bar{\theta}=1$ and \eqref{case3} holds. Thanks to the above claim $\Longrightarrow$, we see \eqref{3205} becomes
\begin{equation*}
\pi^{(j)}= \frac{\rho}{\sigma}\big[ X^{(j)}-\overline{X}\,\big]+\bar{\pi}, \quad 1\leq j\leq n.
\end{equation*}
This $n$-variable system of equations has infinitely many solutions, which are
$$\pi^{(j)}=\pi^{(n)}+\frac{\rho}{\sigma}\Big(X^{(j)}-X^{(n)}\Big), \quad \forall\,\pi^{(n)}\in L^{2,\text{loc}}_{\mathcal{F}}(0,T;\BR),
\quad 1\leq j\leq n-1.$$
They are Nash equilibria in $\seta_1$. The claim (ii) is proved.

Finally, suppose $\bar{\theta}=1$, but \eqref{case3} does not hold.
Then by the above claim $\Longleftarrow$, we get $\rho+\Delta+\frac{\eta}{\psi} \neq 0$, which together with $\bar{\theta}=1$ would contradict \eqref{3206}. So there is no Nash equilibrium in this case and the claim (iii) follows.
%
%Suppose, on the contrary, $\rho+\Delta+\frac{\eta}{\psi} = 0$. Taking it into the definition, we have $f_2(\psi,\eta,\Delta)=-\frac{1}{2}\Delta^2$. Consequently, the unique solution of \eqref{3202} is $(\varphi,\Delta)\equiv(0,0)$, which leads to $ \rho+\frac{\eta}{\psi} =\rho+\Delta+\frac{\eta}{\psi} = 0$. Then, \eqref{3201} becomes \eqref{32002}.
%But since \eqref{case3} does not hold, \eqref{32002} has no solution. This contradiction shows
%$\rho+\Delta+\frac{\eta}{\psi} \neq 0$.
\end{proof}
%
%\begin{remark}
%\blue{
%\eqref{case3} holds if and only if
%\begin{enumerate}[(i)]
%\item $ \mathcal{E}\Big(\int_0^{T} -\rho(t) \dw(t) \Big)e^{-\int_0^T r(t) \dt} $ is a constant;
%\item $ e^{-\int_0^T \BE[ r(t)+ \frac{1}{2}|\rho(t)|^2] \dt} =\widehat{\BE}\big[ e^{-\int_0^T r(t) \dt} \big]$, and $\widehat{\BE}$ is the expectation w.r.t. the risk-neutral probability measure $\widehat{\BP}$, which is defined by
% $$\frac{\dd\widehat{\BP}}{\dd\BP}\bigg|_{\mathcal{F}_T}=\mathcal{E}\Big(\int_0^{T} -\rho(t) \dw(t) \Big).$$
%\end{enumerate}
%That means there exists infinitely many Nash equilibria in $\seta_1$, only if the product of the discount process and the
% Radon-Nikod$\acute{\text{y}}$m derivative process is deterministic at the terminal moment $T$, and the expectation of the terminal discount value w.r.t. the risk-neutral probability measure is equal to the terminal discount value of the average interest rate adjusted by the risk premium.
%}
%\end{remark}

\begin{remark}
Now suppose the interest rate process $r$ is a deterministic function of $t$. Then the unique solution of BSDE \eqref{3201} is $(\psi,\eta)=(\exp\{ \int_t^T r(s) \ds \},0)$. This implies \eqref{case3} does not hold, for otherwise one has $\rho\equiv 0$ in \eqref{32002}. As a consequence, we conclude there is either a unique Nash equilibrium (when $\bar{\theta}<1$) or no Nash equilibrium (when $\bar{\theta}=1$).

When $\bar{\theta}<1$, the unique Nash equilibrium, given by
\begin{equation*}
\begin{aligned}
\pi^{(j,*)}(t) = \Big( \delta_j+ \frac{\bar{\delta}\theta_j}{1-\bar{\theta}} \Big)
\frac{ \rho(t)+\Delta(t) }{\sigma(t)\psi(t)},\,\, 1\leq j\leq n,
\end{aligned}
\end{equation*}
is wealth-independent. This result is consistent with Corollary 2.4 in \cite{Zariphopoulou 1} where all the parameters are constant  and the interest is 0, so that $\Delta\equiv 0$ and $\psi \equiv1$.

On the other hand, if a Nash equilibrium in \eqref{3105} is wealth-independent, we then have $\eta=0$, which together with \eqref{3207} implies $r$ is a deterministic function of $t$. This can also be verified in \cite{Zariphopoulou 2}, \cite{FU Guangxing}, \cite{Zhou Chao} and \cite{Panpan}, where the interest rate is assumed to be constant and the Nash equilibria obtained in these papers are all wealth-independent. Hence the randomness of the interest rate leads to a totally different structure of the Nash equilibria.
\end{remark}
%
%\begin{remark}
%We now suppose the interest process $r$ is a deterministic function of $t$. Then $(Y(t),Z(t))\equiv (e^{\int_{t}^{T}r(s)\ds},0) $ is the unique solution to the following linear BSDE with Lipchitz coefficient:
% \begin{align*}
%\begin{cases}
%\dd Y(t)=-r(t)Y(t)\dt -Z(t) \dw(t),\\
%Y(T)=1.
%\end{cases}
% \end{align*}
%On the other hand, if \eqref{case3} holds true, then there is $\psi$ satisfying \eqref{32002}, which means $(Y (t), Z(t)) \equiv (\psi(t), \rho(t)\psi(t))$
%also satisfies the above BSDE. By the uniqueness of solution, we get $\psi(t)\equiv e^{\int_{t}^{T}r(s)\ds}$ and $\rho(t)\psi(t)\equiv 0$, so $\rho(t)\equiv 0$,
%contradicting our assumption that $\rho$ is not identical to zero. Therefore, we conclude that \eqref{case3} must not hold when the interest process $r$ is a deterministic function of $t$, which indicates that there is either a unique Nash equilibrium (when $\bar{\theta}<1$) or no Nash equilibrium (when $\bar{\theta}=1$).
% In particular, if all the parameters are constant, then
% \begin{enumerate}[(i)]
%\item When $\bar{\theta}<1$, there exists a unique Nash equilibrium, given by
%\begin{equation*}%\label{3105}
%\begin{aligned}
%\pi^{(j,*)}(t) = \Big( \delta_j+ \frac{\bar{\delta}\theta_j}{1-\bar{\theta}} \Big)
%\frac{ \rho }{\sigma },~~ 1\leq j\leq n;
%\end{aligned}
%\end{equation*}
%\item When $\bar{\theta}=1$, there is no Nash equilibrium.
%\end{enumerate}
%This result is consistent with Corollary 2.4 in \cite{Zariphopoulou 1}.
%\end{remark}
%

\section{The CRRA utility case}\label{CRRA}

In this section, we find a Nash equilibrium in $\seta_2$ for game \eqref{203}-\eqref{204} by the BR method. We first find the best response of $\mathbf{A}^{(j)}$ in Theorems \ref{best response 3} and \ref{best response 4}, which is related to a quadratic BSDE \eqref{4101} and a linear BSDE \eqref{4102} respectively.
Because the coefficients of \eqref{4101} depend on $\pi^{(-j)}$, it has unbounded coefficients, making it difficult to obtain its solvability. To let this BSDE admit a solution, we carefully define $\seta_2$.
Then, using a type of decoupling technology, the Nash equilibrium for game \eqref{203}-\eqref{204} is given through solving a series of $1$-dimensional quadratic BSDEs, instead of solving an $n$-dimensional coupled quadratic BSDE.
With the help of Lemma \ref{Linear BSDE 1}, we prove the admissibility of the Nash equilibrium.

%We first find the best response of $\mathbf{A}^{(j)}$. For each fixed $j\in\{1,\cdots,n\}$, the optimization problem for $\mathbf{A}^{(j)}$ is to maximize \eqref{204} with a given admissible strategy vector $\pi^{(-j)}$ subject to \eqref{203}.
Let $\widehat{X}^{(-j)}\triangleq \big[ \prod_{k=1,k\neq j}^{n}X^{(k)}\big]^{\frac{1}{n}}$, then from \eqref{203}, we have
\begin{equation}\label{4100}
\left\{
\begin{aligned}
\frac{\dd \widehat{X}^{(-j)}(t) } {\widehat{X}^{(-j)}(t)}
=\,&\Big\{ \frac{(n-1)r}{n} +\frac{\mu-r}{n}\sum_{k\neq j} \pi^{(k)}(t)
-\frac{\sigma^2}{2n}\sum_{k\neq j} \big[\pi^{(k)}(t)\big]^2
+ \frac{\sigma^2}{2n^2} \Big[ \sum_{k\neq j} \pi^{(k)}(t)\Big]^2 \Big\}\dd t\\
&+\frac{\sigma}{n}\sum_{k\neq j} \pi^{(k)}(t)\dd W(t) ,\\
\widehat{X}^{(-j)}(0) =\,& \hat{x}^{(-j)}\triangleq \Big[ \prod_{k\neq j}x^{(k)} \Big]^{\frac{1}{n}},\quad 1\leq j \leq n.
\end{aligned}
\right.
\end{equation}
And \eqref{204} can be written as
\begin{equation*}
J_2^{(j)}\big( \pi^{(j)},\pi^{(-j)}; x^{(j)},x^{(-j)} \big)=
\begin{cases}
\frac{\delta_j}{\delta_j-1} \BE\big[ {X^{(j)}(T)}^{\beta_{j}} {\widehat{X}^{(-j)}(T)}^{\gamma_{j}} \big], & \delta_i\in(0,1)\cup(1,+\infty),\\
\Big( 1-\frac{\theta_j}{n} \Big)\BE\big[ \log X^{(j)}(T)\big]-\theta_i\BE\big[ \log \widehat{X}^{(-j)}(T)\big], & \delta_j=1,
\end{cases}
\end{equation*}
where for $\theta_j\in[0,1]$, $\delta_j\in (0,+\infty)$, and
$$\beta_{j}\triangleq \Big( 1-\frac{\theta_j}{n} \Big) \Big( 1-\frac{1}{\delta_j} \Big),\quad
\gamma_{j}\triangleq-\theta_j\Big( 1-\frac{1}{\delta_j} \Big).$$
Clearly $\beta_{j}<1$.
For an arbitrary but fixed competitors' strategy vector $\pi^{(-j)}$, the value function for $\mathbf{A}^{(j)}$ in the CRRA utility case is defined by
\begin{equation}\label{4200}
V_2^{(j)}(x^{(j)},x^{(-j)};\pi^{(-j)})\triangleq \sup\limits_{\pi^{(j)}\in \seta_2}
J_2^{(j)}\big( \pi^{(j)},\pi^{(-j)} ; x^{(j)},x^{(-j)} \big),\quad \forall\, (x^{(j)},x^{(-j)})\in\BR^n_{+}.
\end{equation}

For all $1\leq j\leq n$, $(t,P, \Lambda )\in [0,T]\times\BR\times\BR$, and an arbitrary but fixed admissible strategy vector $(\pi^{(1)},\pi^{(2)}, \ldots , \pi^{(n)})\in\seta_2$, we set
\begin{equation*}
h^{(j)}(t,\pi^{(-j)})\triangleq
\frac{1}{2(1-\beta_j)} \Big( \rho(t)+\frac{\gamma_j\sigma}{n} \sum_{k\neq j} \pi^{(k)} \Big)^2
-\frac{\gamma_j\sigma(t)^2}{2n} \sum_{k\neq j} \big[\pi^{(k)} \big]^2
+r(t)(1-\theta_j)(1-\frac{1}{\delta_j})-\frac{\rho(t)^2}{2},
\end{equation*}
\begin{equation*}
H^{(j)}(t, \Lambda,\pi^{(-j)} ) \triangleq \frac{\Lambda^2}{2(1-\beta_j)}
+\frac{1}{1-\beta_j}\Big(\beta_j\rho(t)+\frac{\gamma_j\sigma(t)}{n} \sum_{k\neq j} \pi^{(k)}\Big) \Lambda +h^{(j)}(t,\pi^{(-j)}),
\end{equation*}
\begin{equation*}
F^{(j)}(t,\pi^{(-j)})\triangleq (1-\theta_j)r(t)+\Big(1-\frac{\theta_j}{n} \Big)\frac{\rho(t)^2}{2}
-\theta_j\sigma(t)\Big( \frac{\rho(t)}{n} \sum_{k\neq j} \pi^{(k)} -\frac{\sigma(t)}{2n} \sum_{k\neq j} \big[\pi^{(k)} \big]^2 \Big).
\end{equation*}
In the following, we often omit the argument $t$ for the functions $h^{(j)}$, $H^{(j)}$ and $F^{(j)}$ without causing confusion.

For an arbitrary but fixed admissible strategy vector $(\pi^{(1)},\pi^{(2)}, \ldots , \pi^{(n)})\in\seta_2$, let us introduce the following two BSDEs:
\begin{equation}
\label{4101}
\left\{
\begin{aligned}
\dd P^{(j)}(t)&=-H^{(j)}(\Lambda^{(j)}(t),\pi^{(-j)}(t))\dt +\Lambda^{(j)}(t)\dw(t),\\
P^{(j)}(T)&=0,\quad 1\leq j\leq n,
\end{aligned}
\right.
\end{equation}
and
\begin{equation}
\label{4102}
\left\{
\begin{aligned}
\dd Q^{(j)}(t)&=-F^{(j)}(\pi^{(-j)}(t))\dt +\Gamma^{(j)}(t)\dw(t),\\
Q^{(j)}(T)&=0,\quad 1\leq j\leq n.
\end{aligned}
\right.
\end{equation}
Because some coefficients in \eqref{4101} depend on $\pi^{(-j)}$, they are unbounded.
Thus, BSDE \eqref{4101} is a quadratic BSDE with unbounded coefficients. For this type of BSDE, we cannot use the general result Theorem 2.4 in \cite{Fan Hu Tang} to obtain its solvability.
Therefore, we define the admissible set $\seta_2$ to make sure BSDE \eqref{4101} has a solution, see Definition \ref{admissible set 2}.
From Lemma 3.6 in \cite{Xu arXiv1}, we know BSDE \eqref{4102} admits a unique solution in $L_{\mathcal{F}}^{\infty}(0,T;\BR)\times L_{\mathcal{F}}^{2,\: \mathrm{BMO}}(0,T;\BR)$.

\begin{theorem}[A best response of $\mathbf{A}^{(j)}$ when $\delta_j\neq 1$] \label{best response 3}
In the CRRA utility case, for all $1\leq j\leq n$, for an arbitrary but fixed competitors' strategy vector $\pi^{(-j)}$, a best response of $\mathbf{A}^{(j)}$ whose $\delta_j\neq 1$ is
\begin{equation}\label{4104}
\begin{aligned}
\check{\pi}^{(j)}(t,\pi^{(-j)}) &=
\frac{ \gamma_j \sum_{k\neq j} \pi^{(k)}(t)} {n(1-\beta_j )}
+ \frac{\rho(t)+\Lambda^{(j)}(t)} {(1-\beta_j )\sigma(t)},
\end{aligned}
\end{equation}
and the corresponding value function has an upper bound, which is
\begin{equation*}
V_2^{(j)}(x^{(j)},x^{(-j)};\pi^{(-j)})\leq\frac{\delta_j}{\delta_j-1}
\big[x^{(j)}\big]^{\beta_j} \big[\hat{x}^{(-j)} \big]^{\gamma_j} e^{P^{(j)}(0)},
\end{equation*}
where $(P^{(j)},\Lambda^{(j)})$ is a solution of \eqref{4101}.
\end{theorem}
\begin{proof}
Because in game \eqref{203}-\eqref{204}, all agents are symmetric. We just need to prove that for each fixed $j\in\{1,\cdots,n\}$,
\eqref{4104} is a best response of $\mathbf{A}^{(j)}$ whose $\delta_j\neq 1$.
For an arbitrary but fixed competitors' strategy vector $\pi^{(-j)}$, and any $\pi^{(j)}$ such that $(\pi^{(j)},\pi^{(-j)})\in \seta_2$,
applying It\^{o}'s formula to $ \big[X^{(j)}\big]^{\beta_j} \big[\widehat{X}^{(-j)}\big]^{\gamma_j} e^{P^{(j)}}$, where $X^{(j)}$ is the solution of \eqref{203} and $\widehat{X}^{(-j)}$ is the solution of \eqref{4100}, we get
\begin{equation*}
\begin{aligned}
&\big[X^{(j)}(t)\big]^{\beta_j} \big[\widehat{X}^{(-j)}(t)\big]^{\gamma_j}e^{P^{(j)}(t)}
-\big[x^{(j)} \big]^{\beta_j} \big[\hat{x}^{(-j)} \big]^{\gamma_j} e^{P^{(j)}(0)}\\
=\,&-\frac{ \beta_j(1- \beta_j)\sigma^2}{2} \int_0^t \big[X^{(j)}\big]^{\beta_{j}} \big[\widehat{X}^{(-j)}\big]^{\gamma_{j}} e^{P^{(j)}}
\big[ \pi^{(j)}-\check{\pi}^{(j)}(\pi^{(-j)}) \big]^2\ds\\
&+\int_0^t \big[X^{(j)}\big]^{\beta_{j}} \big[\widehat{X}^{(-j)}\big]^{\gamma_{j}} e^{P^{(j)}}
\Big( \beta_j\sigma\pi^{(j)} +\frac{\gamma_j\sigma}{n}\sum_{k\neq j} \pi^{(k)} +\Lambda^{(j)} \Big) \dw(s).
\end{aligned}
\end{equation*}
Noting that the processes $X^{(j)}$ and $\widehat{X}^{(-j)}$ are continuous and strictly positive, so almost surely bounded on $[0, T]$ away from $0$. Therefore, the stochastic integral in the above equation is a local martingale. Hence, there exists a nondecreasing sequence of $\{\mathcal{F}_t\}_{t\geq0}$-stopping times $\tau_k$ satisfying $\tau_k\rightarrow T$ as $k\rightarrow \infty$ such that, for all $k$,
\begin{equation*}
\begin{aligned}
&\BE\Big[ \big[X^{(j)}(T\wedge\tau_k)\big]^{\beta_j} \big[\widehat{X}^{(-j)}(T\wedge\tau_k)\big]^{\gamma_j} e^{P^{(j)}(T\wedge\tau_k)} \Big]-\big[ x^{(j)} \big]^{\beta_j} \big[\hat{x}^{(-j)} \big]^{\gamma_j} e^{P^{(j)}(0)}\\
=\,& -\frac{ \beta_j(1- \beta_j)\sigma^2}{2} \BE \Big[ \int_0^{T\wedge\tau_k}
\big[X^{(j)}\big]^{\beta_{j}} \big[\widehat{X}^{(-j)}\big]^{\gamma_{j}} e^{P^{(j)}}
\big[ \pi^{(j)}-\check{\pi}^{(j)}(\pi^{(-j)}) \big]^2 \ds\Big].
\end{aligned}
\end{equation*}
By Definition \ref{admissible set 2} and the equivalence between uniform integrability and $\mathcal{L}^1$ convergence,
the expectation on the left-hand side is convergent as $k \rightarrow \infty$.
Meanwhile, by the monotone convergence theorem, the right-hand side also converges as $k \rightarrow \infty$. Thus, we obtain
\begin{equation*}
\begin{aligned}
&J_2^{(j)}( \pi^{(j)},\pi^{(-j)}; x^{(j)},x^{(-j)} )=
\frac{\delta_j}{\delta_j-1} \BE\Big[ \big[X^{(j)}(T)\big]^{\beta_j} \big[\widehat{X}^{(-j)}(T)\big]^{\gamma_j} e^{P^{(j)}(T)} \Big]\\
=\,& \frac{\delta_j}{\delta_j-1}\big[ x^{(j)} \big]^{\beta_j} \big[\hat{x}^{(-j)} \big]^{\gamma_j} e^{P^{(j)}(0)}
-\frac{ \delta_j\beta_j(1- \beta_j)\sigma^2}{2(\delta_j-1)} \BE \Big[ \int_0^{T}
\big[X^{(j)}\big]^{\beta_{j}} \big[\widehat{X}^{(-j)}\big]^{\gamma_{j}} e^{P^{(j)}}
\big[ \pi^{(j)}-\check{\pi}^{(j)}(\pi^{(-j)}) \big]^2 \ds\Big]\\
\leq \,& \frac{\delta_j}{\delta_j-1}\big[ x^{(j)} \big]^{\beta_j} \big[\hat{x}^{(-j)} \big]^{\gamma_j} e^{P^{(j)}(0)},
\end{aligned}
\end{equation*}
and the equality holds if and only if $\pi^{(j)}=\check{\pi}^{(j)}(\pi^{(-j)})$. The proof is complete.
\end{proof}

\begin{theorem}[A best response strategy of $\mathbf{A}^{(j)}$ when $\delta_j= 1$] \label{best response 4}
In the CARA utility case, for all $1\leq j\leq n$, for an arbitrary but fixed competitors' strategy vector $\pi^{(-j)}$, a best response of $\mathbf{A}^{(j)}$ whose $\delta_j= 1$ is
\begin{equation}\label{4105}
\check{\pi}^{(j)}(t)=\frac{\rho(t)}{\sigma(t)},
\end{equation}
and the corresponding value function has an upper bound, which is
\begin{equation*}
V_2^{(j)}(x^{(j)},x^{(-j)};\pi^{(-j)})\leq\Big(1-\frac{\theta_j}{n}\Big)\log x^{(j)}-\theta_j \log \hat{x}^{(-j)} +Q^{(j)}(0),
\end{equation*}
where $(Q^{(j)},\Gamma^{(j)})$ is the unique solution of \eqref{4102}.
\end{theorem}
\begin{proof}
Because in game \eqref{203}-\eqref{204}, all agents are symmetric. We just need to prove that for each fixed $j\in\{1,\cdots,n\}$,
\eqref{4105} is a best response of $\mathbf{A}^{(j)}$ whose $\delta_j=1$.
For an arbitrary but fixed competitors' strategy vector $\pi^{(-j)}$, and any $\pi^{(j)}$ such that $(\pi^{(j)},\pi^{(-j)})\in \seta_2$,
applying It\^{o}'s formula to
$ \big(1-\frac{\theta_j}{n}\big)\log X^{(j)}-\theta_j \log \widehat{X}^{(-j)} +Q^{(j)}$, where $X^{(j)}$ is the solution of \eqref{203} and $\widehat{X}^{(-j)}$ is the solution of \eqref{4100}, we get
\begin{equation*}
\begin{aligned}
&\Big(1-\frac{\theta_j}{n}\Big)\log X^{(j)}(t)-\theta_j \log \widehat{X}^{(-j)}(t) +Q^{(j)}(t)
-\Big[ \Big(1-\frac{\theta_j}{n}\Big)\log x^{(j)}-\theta_j \log \hat{x}^{(-j)} +Q^{(j)}(0)\Big]\\
=\,&-\frac{1}{2}\Big(1-\frac{\theta_j}{n}\Big) \int_0^t \big[\sigma\pi^{(j)}-\rho\big]^2 \ds
+\int_0^t \Big( \sigma\Big[\pi^{(j)}-\frac{\theta_j}{n}\sum_{k=1}^n \pi^{(k)} \Big]+\Gamma^{(j)}\Big) \dw (s).
\end{aligned}
\end{equation*}
Because the stochastic integral in the above equation is a martingale, we obtain
\begin{equation*}
\begin{aligned}
&\BE\Big[\Big(1-\frac{\theta_j}{n}\Big)\log X^{(j)}(T)-\theta_j \log \widehat{X}^{(-j)}(T) \Big]\\
=\,&\Big(1-\frac{\theta_j}{n}\Big)\log x^{(j)}-\theta_j \log \hat{x}^{(-j)} +Q^{(j)}(0)-\frac{1}{2}\Big(1-\frac{\theta_j}{n}\Big)
\BE\Big[ \int_0^T \big[\sigma\pi^{(j)}-\rho\big]^2 \ds\Big]\\
\leq \,&\Big(1-\frac{\theta_j}{n}\Big)\log x^{(j)}-\theta_j \log \hat{x}^{(-j)} +Q^{(j)}(0).
\end{aligned}
\end{equation*}
In addition, the equality holds if and only if $\pi^{(j)}=\check{\pi}^{(j)}$. The proof is complete.
\end{proof}

\begin{remark}
When $\delta_j=1$, we have $\beta_j=\gamma_j=0$ and the unique solution of \eqref{4101} is $(P^{(j)},\Lambda^{(j)})=(0,0)$. In this case, the best response \eqref{4104} is coincidence with \eqref{4105}. Therefore, the best response of $\mathbf{A}^{(j)}$ in the CRRA utility case can be
written in a unified form, which is given by \eqref{4104}.
\end{remark}

To solve the Nash equilibrium in $\seta_2$ for game \eqref{203}-\eqref{204}, for convenience, we rewrite \eqref{4101}.
Let $(\widehat{P}^{(j)},\widehat{\Lambda}^{(j)})=(\delta_j P^{(j)},\delta_j\Lambda^{(j)})$.
Then, for an arbitrary but fixed admissible strategy vector $(\pi^{(1)},\pi^{(2)}, \ldots , \pi^{(n)})\in\seta_2$,
$(\widehat{P}^{(j)},\widehat{\Lambda}^{(j)}) \in L_{\mathcal{F}}^{\infty}(0,T;\BR)\times L_{\mathcal{F}}^{2,\: \mathrm{BMO}}(0,T;\BR)$, $1\leq j\leq n$,
are solutions of the following BSDEs:
\begin{equation}
\label{4201}
\left\{
\begin{aligned}
\dd \widehat{P}^{(j)}(t)&=-\widehat{H}^{(j)}(\widehat{\Lambda}^{(j)}(t),\pi^{(-j)}(t))\dt +\widehat{\Lambda}^{(j)}(t)\dw(t),\\
\widehat{P}^{(j)}(T)&=0, \quad 1\leq j\leq n,
\end{aligned}
\right.
\end{equation}
where for all $1\leq j\leq n$, $(t, \Lambda)\in [0,T]\times\BR$, and an arbitrary but fixed admissible strategy vector $(\pi^{(1)},\pi^{(2)}, \ldots , \pi^{(n)})\in\seta_2$,
\begin{equation*}
\widehat{H}^{(j)}(t, \Lambda,\pi^{(-j)} ) \triangleq \frac{\Lambda^2}{2(1-\beta_j)\delta_j}
+\frac{1}{1-\beta_j}\Big(\beta_j\rho(t)+\frac{\gamma_j\sigma(t)}{n} \sum_{k\neq j} \pi^{(k)}\Big) \Lambda +\delta_j h^{(j)}(t,\pi^{(-j)}).
\end{equation*}

For all $1\leq j\leq n$, $\delta_j>0$ and $\theta_j\in[0,1]$, we define the following constants:
%$1-\beta_j=\frac{1}{\delta_j}-\frac{\gamma_j}{n}$
$$
\overline{\delta^2}\triangleq\frac{1}{n}\sum_{j=1}^n \big[\delta_j^2\big], \quad
\overline{\delta^2\gamma}\triangleq\frac{1}{n}\sum_{j=1}^n \big[\delta_j^2\gamma_j\big], \quad
\overline{\delta^3\gamma}\triangleq\frac{1}{n}\sum_{j=1}^n \big[\delta_j^3\gamma_j\big], \quad
\overline{\delta^2\gamma^2}\triangleq\frac{1}{n}\sum_{j=1}^n \big[\delta_j^2\gamma_j^2\big],
$$
$$
\overline{\delta\gamma}\triangleq\frac{1}{n}\sum_{j=1}^n \big[\delta_j\gamma_j\big], \quad
\bar{\lambda}\triangleq \frac{1}{1-\overline{\delta\gamma}},\quad
C_1^{(j)}\triangleq\delta_j\gamma_j\bar{\lambda}\bar{\delta}+\delta_j,\quad
\overline{C}_1\triangleq \frac{1}{n} \sum_{j=1}^n C_1^{(j)},\quad
\overline{C_1^2}\triangleq\frac{1}{n}\sum_{j=1}^n \big[C_1^2\big],
$$
$$
C_2^{(j)}\triangleq \delta_j^2(\gamma_j\bar{\lambda}\bar{\delta}+1)^2 -\delta_j\gamma_j\bar{\lambda}\bar{\delta}-\delta_j
=C_1^{(j)}(C_1^{(j)}-1),\quad
C_3^{(j)}\triangleq(\delta_j-1)(1-\theta_j)+\delta_j\gamma_j(\bar{\lambda}\bar{\delta}-1)=C_1^{(j)}-1.$$

For each $j\in\{1,\cdots, n\}$, we introduce a 1-dimensional quadratic BSDE:
\begin{equation}
\label{4202}
\left\{
\begin{aligned}
\dd Y^{(j)}(t)&=-\Big[\frac{1}{2}\big[Z^{(j)}(t)\big]^2+(C_1^{(j)}-1)\rho Z^{(j)}(t)+\frac{1}{2}C_2^{(j)}\rho^2+C_3^{(j)}r \Big]\dt +Z^{(j)}(t)\dw(t),\\
Y^{(j)}(T)&=0.
\end{aligned}
\right.
\end{equation}
From Theorem 2.4 in \cite{Fan Hu Tang}, the quadratic BSDE \eqref{4202} admits a unique solution in $L_{\mathcal{F}}^{\infty}(0,T;\BR)\times L_{\mathcal{F}}^{2,\: \mathrm{BMO}}(0,T;\BR)$.
Let $(\widehat{P}^{(j)},\widehat{\Lambda}^{(j)})$ and $(Y^{(j)},Z^{(j)})$ be the solutions of BSDE \eqref{4201} and \eqref{4202}, respectively. We define the following processes:
$$
\overline{P}\triangleq \frac{1}{n} \sum_{j=1}^n \widehat{P}^{(j)},\quad
\overline{\Lambda}\triangleq \frac{1}{n} \sum_{j=1}^n \widehat{\Lambda}^{(j)},\quad
\overline{\Lambda^2}\triangleq \frac{1}{n} \sum_{j=1}^n \big[\widehat{\Lambda}^{(j)}\big]^2,\quad
\overline{\delta \Lambda}\triangleq\frac{1}{n}\sum_{j=1}^n \big[\delta_j \widehat{\Lambda}^{(j)}\big],
$$
$$
\overline{\delta\gamma \Lambda}\triangleq\frac{1}{n}\sum_{j=1}^n \big[\delta_j\gamma_j\widehat{\Lambda}^{(j)}\big],\quad
\overline{Z}\triangleq \frac{1}{n} \sum_{j=1}^n Z^{(j)},\quad
\overline{Z^2}\triangleq \frac{1}{n} \sum_{j=1}^n \big[Z^{(j)}\big]^2,\quad
\overline{C_1 Z}\triangleq\frac{1}{n}\sum_{j=1}^n \big[C_1^{(j)} Z^{(j)}\big].
$$

\begin{theorem}[Nash equilibrium in the CRRA utility case] \label{NE 2}
There exists a Nash equilibrium in $\seta_2$ for game \eqref{203}-\eqref{204}, which is given by
\begin{equation}\label{4203}
\begin{aligned}
\pi^{(j,*)}(t) = \frac{ C_1^{(j)}\rho(t)+Z^{(j)}(t) }{\sigma(t)},\quad 1\leq j\leq n,
\end{aligned}
\end{equation}
where $(Y^{(j)},Z^{(j)})$ is the unique solution of BSDE \eqref{4202}. Furthermore, the equilibrium value function of $\mathbf{A}^{(j)}$ is
\begin{equation}\label{4204}
V_2^{(j)}(x^{(j)},x^{(-j)};\pi^{(-j,*)})=
\begin{cases}
\frac{\delta_j}{\delta_j-1}\big[x^{(j)}\big]^{\beta_j} \big[\hat{x}^{(-j)} \big]^{\gamma_j} e^{P^{(j)}(0)}, & \delta_i\in(0,1)\cup(1,+\infty),\bigskip\\
\Big(1-\frac{\theta_j}{n}\Big)\log x^{(j)}-\theta_j \log \hat{x}^{(-j)} +Q^{(j)}(0), & \delta_j=1,
\end{cases}
\end{equation}
where $(P^{(j)},\Lambda^{(j)})$ and $(Q^{(j)},\Gamma^{(j)})$ are solutions of the following BSDEs:
\begin{equation*}
\left\{
\begin{aligned}
\dd P^{(j)}(t)=\,&-\Big[ \frac{\big[\Lambda^{(j)}(t)\big]^2}{2(1-\beta_j)}
+\frac{1}{1-\beta_j}\Big(\beta_j\rho+\gamma_j\overline{C}_1\rho+\gamma_j\overline{Z}(t)
-\frac{\gamma_j\big[C_1^{(j)}\rho+Z^{(j)}(t)\big]}{n} \Big) \Lambda^{(j)}(t)\\
&+\frac{1}{2(1-\beta_j)} \Big( \rho+\gamma_j\overline{C}_1\rho+\gamma_j\overline{Z}(t)
-\frac{\gamma_j\big[C_1^{(j)}\rho+Z^{(j)}(t)\big]}{n} \Big)^2
+r(1-\theta_j)(1-\frac{1}{\delta_j})\\
&-\frac{\rho^2}{2}
-\frac{\gamma_j}{2} \Big( \overline{C_1^2}\rho^2+\overline{Z^2}(t)+2\rho\overline{ C_1Z}(t)
-\frac{ \big[C_1^{(j)}\rho+Z^{(j)}(t)\big]^2 } {n}\Big)\Big]\dt +\Lambda^{(j)}(t)\dw(t),\\
P^{(j)}(T)=\,&0,\quad 1 \leq j\leq n,
\end{aligned}
\right.
\end{equation*}
and
\begin{equation*}
\left\{
\begin{aligned}
\dd Q^{(j)}(t)=\,&-\Big[ (1-\theta_j)r+\Big(1-\frac{\theta_j}{n} \Big)\frac{\rho^2}{2}
-\theta_j\rho\Big( \overline{C}_1\rho+\overline{Z}(t) -\frac{C_1^{(j)}\rho+Z^{(j)}(t)}{n} \Big)\\
& +\frac{\theta_j}{2} \Big( \overline{C_1^2}\rho^2+\overline{Z^2}(t)+2\rho\overline{ C_1Z}(t)
-\frac{ \big[C_1^{(j)}\rho+Z^{(j)}(t)\big]^2 } {n}\Big) \Big]\dt +\Gamma^{(j)}(t)\dw(t),\\
Q^{(j)}(T)=\,&0, \quad 1 \leq j\leq n.
\end{aligned}
\right.
\end{equation*}
\end{theorem}
\begin{proof}
To obtain the Nash equilibrium, we solve the $n$-variable system of equations \eqref{205} where $\check{\pi}^{(j)}(\pi^{(-j)} )$ is given by \eqref{4104}.
Using the identity
$$\frac{1}{n}\sum_{k\neq j} \pi^{(k)}=\bar{\pi}-\frac{1}{n}\pi^{(j)},$$
we see \eqref{205} is equivalent to the following $n$-variable system of equations:
\begin{equation}\label{4205}
\begin{aligned}
\pi^{(j)}= \frac{\delta_j\rho+\widehat{\Lambda}^{(j)}}{\sigma}+ \delta_j\gamma_j \bar{\pi}, \quad 1\leq j\leq n,
\end{aligned}
\end{equation}
where $(\widehat{P}^{(j)},\widehat{\Lambda}^{(j)})$ is a solution of \eqref{4201}.
Then, averaging over $j=1,2, \ldots n$, we attain
$$\bar{\pi}=\frac{\bar{\lambda}( \bar{\delta}\rho+ \overline{\Lambda} \,)}{\sigma}.$$
Thus, \eqref{4205} is equivalent to the following $n$-variable system of equations:
\begin{equation}\label{4206}
\begin{aligned}
\pi^{(j)}=\frac{1}{\sigma}\big[ C_1^{(j)}\rho+\widehat{\Lambda}^{(j)}+ \delta_j\gamma_j\bar{\lambda}\overline{\Lambda}\,\big] ,
\quad 1\leq j\leq n.
\end{aligned}
\end{equation}

For any solution of \eqref{4206}, we have
\begin{equation}\label{4207}
\begin{aligned}
\frac{\sigma}{n}\sum_{k\neq j} \pi^{(k)}=\,& \sigma\Big(\bar{\pi}-\frac{1}{n}\pi^{(j)}\Big)=
\Big( \bar{\lambda}\bar{\delta}-\frac{C_1^{(j)}}{n}\Big)\rho
-\frac{1}{n}\widehat{\Lambda}^{(j)}
+\Big( 1-\frac{\delta_j\gamma_j}{n}\Big)\bar{\lambda} \overline{\Lambda},\\
\frac{\sigma^2}{n}\sum_{k\neq j} {\big[\pi^{(k)}\big]}^2 =\,&
\sigma^2\Big( \sum_{k=1}^n {\big[\pi^{(k)}\big]}^2- {\big[\pi^{(j)}\big]}^2 \Big)
=\overline{\Lambda^2}+\bar{\lambda}^2\overline{\delta^2\gamma^2}\,\overline{\Lambda}^2
+\big( \,\overline{\delta^2}+\bar{\lambda}^2\bar{\delta}^2\overline{\delta^2\gamma^2}
+2\bar{\lambda}\bar{\delta}\, \overline{\delta^2\gamma}\, \big)\rho^2\\
&+2\bar{\lambda}\overline{\delta\gamma\Lambda}\,\overline{\Lambda}+2\bar{\lambda}\bar{\delta}\rho\overline{\delta\gamma\Lambda}
+2\rho\overline{\delta\Lambda}
+2\bar{\lambda}\big(\,\overline{\delta^2\gamma}+\bar{\lambda}\bar{\delta}\,\overline{\delta^2\gamma^2}\,\big)\rho\overline{\Lambda}
-\frac{1}{n}\Big(\big[\widehat{\Lambda}^{(j)}\big]^2+\delta_j^2\gamma_j^2\bar{\lambda}^2\overline{\Lambda}^2\\
&+\delta_j^2(\gamma_j\bar{\lambda}\bar{\delta}+1)^2\rho^2
+2\delta_j\gamma_j\bar{\lambda}\widehat{\Lambda}^{(j)}\overline{\Lambda}
+2\delta_j(\gamma_j\bar{\lambda}\bar{\delta}+1)\rho\widehat{\Lambda}^{(j)}
+2\delta_j^2\gamma_j\bar{\lambda}(\gamma_j\bar{\lambda}\bar{\delta}+1)\rho\overline{\Lambda} \Big).
\end{aligned}
\end{equation}
Substituting \eqref{4207} into \eqref{4201}, we get
\begin{equation}\label{4208}
\left\{
\begin{aligned}
\dd \widehat{P}^{(j)}(t)=\,&-\Big[ \frac{[\widehat{\Lambda}^{(j)}(t)]^2}{2}
+ \delta_j\gamma_j\bar{\lambda} \widehat{\Lambda}^{(j)}(t)\overline{\Lambda}(t)
+ \delta_j\gamma_j\bar{\lambda}^2 \big( \delta_j\gamma_j-\overline{\delta^2\gamma^2} \,\big) \frac{[\overline{\Lambda}(t)]^2}{2}
-\frac{\delta_j\gamma_j}{2}\Big( \overline{\Lambda^2}(t)\\
&+2\bar{\lambda} \overline{\delta\gamma \Lambda}(t) \overline{\Lambda}(t)
+2\rho\big(\bar{\lambda}\bar{\delta}\, \overline{\delta\gamma \Lambda}(t) + \overline{\delta \Lambda}(t)\big)\Big)
+ \big( \delta_j-1+\delta_j\gamma_j\bar{\lambda}\bar{\delta} \big) \rho \widehat{\Lambda}^{(j)}(t)\\
&+ \delta_j\gamma_j\bar{\lambda} \big( \delta_j(\gamma_j\bar{\lambda}\bar{\delta}+1) - \overline{\delta^2\gamma}
-\bar{\lambda}\bar{\delta}\,\overline{\delta^2\gamma^2} \, \big) \rho\overline{\Lambda}(t)
+\frac{\rho^2}{2}\Big(\delta_j^2(1+\gamma_j\bar{\lambda}\bar{\delta} )^2-\delta_j\\
&-\delta_j\gamma_j \big(\overline{\delta^2}+2\bar{\lambda}\bar{\delta}\,\overline{\delta^2\gamma}
+\bar{\lambda}^2\bar{\delta}^2\overline{\delta^2\gamma^2}\, \big) \Big)
+r(1-\theta_j)(\delta_j-1) \Big] \dt +\widehat{\Lambda}^{(j)}(t)\dw(t),\\
\widehat{P}^{(j)}(T)=\,&0, \quad 1\leq j\leq n.
\end{aligned}
\right.
\end{equation}
Then, $(\overline{P},\overline{\Lambda})$ satisfies the following BSDE:
\begin{equation}\label{4209}
\left\{
\begin{aligned}
\dd \overline{P}(t)=\,&-\Big[ \frac{1}{2\bar{\lambda}} \Big( \overline{\Lambda^2}(t)
+2\bar{\lambda} \overline{\delta\gamma \Lambda}(t) \overline{\Lambda}(t)
+2\rho\big(\bar{\lambda}\bar{\delta}\, \overline{\delta\gamma \Lambda}(t) + \overline{\delta \Lambda}(t)\big)\Big) \\ &
+ \bar{\lambda}\overline{\delta^2\gamma^2}\frac{[\overline{\Lambda}(t)]^2}{2}
+\big( \overline{\delta^2\gamma}+\bar{\lambda}\bar{\delta}\,\overline{\delta^3\gamma}-1\big)\rho\overline{\Lambda}(t)
-\frac{\bar{\delta}\rho^2}{2}\\
&+\big(\overline{\delta\delta}+2\bar{\lambda}\bar{\delta}\,\overline{\delta^2\gamma}
+\bar{\lambda}^2\bar{\delta}^2\overline{\delta^2\gamma^2}\, \big)\frac{\rho^2}{2\bar{\lambda}}
+\frac{(\bar{\lambda}\bar{\delta}-1)r}{\bar{\lambda}} \Big] \dt
+\overline{\Lambda}(t)\dw(t),\\
\overline{P}(T)=\,&0.
\end{aligned}
\right.
\end{equation}
Let $(Y^{(j)},Z^{(j)})=(\widehat{P}^{(j)}+ \delta_j\gamma_j\bar{\lambda}\overline{P},
\widehat{\Lambda}^{(j)}+ \delta_j\gamma_j\bar{\lambda}\overline{\Lambda})$. From \eqref{4208}-\eqref{4209}, we have
$(Y^{(j)},Z^{(j)})$ satisfies \eqref{4202}. And \eqref{4206} becomes \eqref{4203}. Taking \eqref{4203} into \eqref{4101}-\eqref{4102},
the equilibrium value function of $\mathbf{A}^{(j)}$ is given by \eqref{4204}.

To end this proof, it only remains to show that the Nash equilibrium given in \eqref{4203} is admissible.
For every $1\leq j\leq n$, because $Z^{(j)}\in L^{2,\text{BMO}}_{\mathcal{F}}(0,T;\BR)$, it is obvious $\pi^{(j,*)}\in L^{2,\text{BMO}}_{\mathcal{F}}(0,T;\BR)$.

SDE \eqref{203} with the strategy $\pi^{(j)}=\pi^{(j,*)}$ is:
\begin{equation*}
\left\{
\begin{aligned}
\frac{\dd X^{(j)}(t)}{X^{(j)}(t)}=\,&\big[ r+\rho( C_1^{(j)}\rho+Z^{(j)}) \big]\dt
+ ( C_1^{(j)}\rho+Z^{(j)}) \dw(t),\\
X^{(j)}(0)=\,&x^{(j)}\in \BR_{+},\quad 1\leq j\leq n.
\end{aligned}
\right.
\end{equation*}
The above equation is a linear SDE of $X^{(j)}$, and it has a unique strong solution in $L^{0}_{\mathcal{F}}(C(0,T);\BR)$, which is
$$X^{(j)}(t)=x^{(j)}\exp\Big\{ \int_0^t \Big(r+\rho( C_1^{(j)}\rho+Z^{(j)})-\frac{1}{2}( C_1^{(j)}\rho+Z^{(j)})^2 \Big)\ds
+\int_0^t ( C_1^{(j)}\rho+Z^{(j)})\dw (s) \Big\}.$$

To prove the solvability of \eqref{4101} with the Nash equilibrium given in \eqref{4203}, we first use exponential transformation to transform \eqref{4101} to a linear BSDE with unbounded coefficients. Let $(\widetilde{P}^{(j)},\widetilde{\Lambda}^{(j)})=(e^{\frac{P^{(j)}}{1-\beta_j}}, \frac{\Lambda}{1-\beta_j}e^{\frac{P^{(j)}}{1-\beta_j}})$,
$1\leq j\leq n$.
Then $(\widetilde{P}^{(j)},\widetilde{\Lambda}^{(j)})$ satisfies the following linear BSDE:
\begin{equation}
\label{4103}
\left\{
\begin{aligned}
\dd \widetilde{P}^{(j)}(t)&=-\Big[ \frac{h^{(j)}(\pi^{(-j,*)})}{1-\beta_j}\widetilde{P}^{(j)}(t)
+\frac{1}{1-\beta_j}\Big(\beta_j\rho+\frac{\gamma_j\sigma}{n} \sum_{k\neq j} \pi^{(k,*)}\Big)\widetilde{\Lambda}^{(j)}(t)\Big] \dt +\widetilde{\Lambda}^{(j)}(t)\dw(t),\\
\widetilde{P}^{(j)}(T)&=1.
\end{aligned}
\right.
\end{equation}
Because $\pi^{(j,*)}\in L^{2,\text{BMO}}_{\mathcal{F}}(0,T;\BR)$ for all $1\leq j\leq n$, $\mathcal{E}\big( \int_0^{\cdot} \frac{1}{1-\beta_j}\big(\beta_j\rho+\frac{\gamma_j\sigma}{n} \sum_{k\neq j} \pi^{(k,*)}\big) \dw (s) \big)$ is a uniformly integrable martingale on $[0,T]$. Then for each $1\leq j\leq n$, we can define the probability measure $\widetilde{\BP}^{(j)}$ by
\begin{equation*}
\frac{\dd \widetilde{\BP}^{(j)}} {\dd \BP}\bigg|_{\mathcal{F}_T}=\mathcal{E}\Big( \int_0^T \frac{1}{1-\beta_j}\Big(\beta_j\rho+\frac{\gamma_j\sigma}{n} \sum_{k\neq j} \pi^{(k,*)}\Big) \dw (s) \Big).
\end{equation*}
Consequently,
$$ \widetilde{W}^{(j)}(t) \triangleq W(t) -\int_0^t \frac{1}{1-\beta_j}\Big(\beta_j\rho+\frac{\gamma_j\sigma}{n} \sum_{k\neq j} \pi^{(k,*)}\Big) \ds$$
is a Brownian motion under the probability $\widetilde{\BP}^{(j)}$.
From \eqref{4202}, for all $1\leq j\leq n$, when $\gamma_j\neq0$, under the probability $\widetilde{\BP}^{(j)}$ we have
\begin{equation*}
\left\{
\begin{aligned}
\dd \frac{\gamma_j}{n(1-\beta_j)}\sum_{k\neq j}Y^{(k)}(t)=\,&\Big\{\frac{h^{(j)}(\pi^{(-j,*)})}{1-\beta_j}
+\frac{\gamma_j^2}{2n^2(1-\beta_j)^2}\Big[\sum_{k\neq j} Z^{(k)}(t) \Big]^2 +L^{(j)}(t)\Big\}\dt\\
&+\frac{\gamma_j}{n(1-\beta_j)}\sum_{k\neq j} Z^{(k)}(t) \dd\widetilde{W}^{(j)}(t),\\
\frac{\gamma_j}{n(1-\beta_j)}\sum_{k\neq j}Y^{(k)}(T)=\,&0,\quad 1\leq j\leq n,
\end{aligned}
\right.
\end{equation*}
where \begin{equation*}
\begin{aligned}L^{(j)}(t)\triangleq \,& \frac{\rho^2(t)} {2(1-\beta_j)} \Big[ 1+\frac{\gamma_j}{n}\sum_{k\neq j}C_1^{(k)}
-\frac{1}{1-\beta_j}\Big(1+\frac{\gamma_j}{n}\sum_{k\neq j}C_1^{(k)} \Big)^2 \Big]\\
&-\frac{r(t)} {1-\beta_j}\Big[ (1-\theta_j)\Big(1-\frac{1}{\delta_j}\Big) +\frac{\gamma_j}{n}\sum_{k\neq j}C_3^{(k)}\Big]
\in L^{\infty}_{\mathcal{F}}(0,T;\BR).
\end{aligned}
\end{equation*}
For $1\leq j\leq n$, let $\widetilde{\BE}^{(j)}$ be the expectation w.r.t. $\widetilde{\BP}^{(j)}$. For all $1\leq j\leq n$, we obtain
\begin{equation*}
\begin{aligned}
\int_t^T \frac{h^{(j)}(\pi^{(-j,*)})}{1-\beta_j} \ds =\,&
-\int_t^T\frac{\gamma_j^2}{2n^2(1-\beta_j)^2}\Big[\sum_{k\neq j} Z^{(k)} \Big]^2 \ds
-\int_t^T \frac{\gamma_j}{n(1-\beta_j)}\sum_{k\neq j} Z^{(k)} \dd\widetilde{W}^{(j)}(s)\\
&-\int_t^T L^{(j)} \ds -\frac{\gamma_j}{n(1-\beta_j)}\sum_{k\neq j}Y^{(k)}(t), ~~\forall\;t\in[0,T].
\end{aligned}
\end{equation*}
Then, for all $1\leq j\leq n$ and $t\in[0,T]$, we get
\begin{equation*}
\widetilde{\BE}^{(j)}_t\Big[\exp\Big( \int_t^T \frac{h^{(j)}(\pi^{(-j,*)})}{1-\beta_j} \ds \Big)\Big]
=\widetilde{\BE}^{(j)}_t\Big[ \frac{M^{(j)}(T)}{M^{(j)}(t)}\exp\Big(-\int_t^T L^{(j)} \ds -\frac{\gamma_j}{n(1-\beta_j)}\sum_{k\neq j}Y^{(k)}(t) \Big)\Big],
\end{equation*}
where $$M^{(j)}(t)\triangleq \mathcal{E}\Big( -\int_0^t \frac{\gamma_j}{n(1-\beta_j)}\sum_{k\neq j} Z^{(k)} \dd\widetilde{W}^{(j)}(s)\Big).$$
Note that $\int_0^{\cdot} \frac{\gamma_j}{n(1-\beta_j)}\sum_{k\neq j} Z^{(k)} \dd W(s)$ is a BMO martingale under $\BP$,
so $\int_0^{\cdot} \frac{\gamma_j}{n(1-\beta_j)}\sum_{k\neq j} Z^{(k)} \dd\widetilde{W}^{(j)}(s)$ is a BMO martingale under $\widetilde{\BP}^{(j)}$ and $M^{(j)}(t)$ is a uniformly integrable martingale under $\widetilde{\BP}^{(j)}$. Therefore, we have
$$\widetilde{\BE}^{(j)}_t\Big[\frac{M^{(j)}(T)}{M^{(j)}(t)} \Big]= \frac{ \widetilde{\BE}^{(j)}_t\big[M^{(j)}(T)\big] } {M^{(j)}(t)}
=\frac{M^{(j)}(t)}{M^{(j)}(t)}=1, ~~\forall\;t\in[0,T].$$
Combining the boundedness of $L^{(j)}$ and $Y^{(j)}$, we obtain that there are positive constants $c_{1}^{(j)}$, $c_{2}^{(j)}$ such that
\begin{equation}\label{a1}
c_1^{(j)}\leq \widetilde{\BE}^{(j)}_t\Big[\exp\Big( \int_t^T \frac{h^{(j)}(\pi^{(-j,*)})}{1-\beta_j} \ds \Big)\Big]\leq c_2^{(j)},~~\forall\;t\in[0,T], ~~1\leq j\leq n.
\end{equation}
When $\gamma_j=0$, $h^{(j)}(\pi^{(-j,*)})\in L^{\infty}_{\mathcal{F}}(0,T;\BR)$, the above inequality also holds. From the definition of $h^{(j)}$, we have
$$|h^{(j)}(\pi^{(-j,*)})|\leq K^{(j)}\Big(1+\sum_{k\neq j} \big[Z^{(k)}\big]^2 \Big), ~~1\leq j\leq n,$$
where $K^{(j)}$ are positive constants. Because $Z^{(j)}\in L^{2,\text{BMO}}_{\mathcal{F}}(0,T;\BR)$ for all $1\leq j\leq n$, there are positive constants $c_{3}^{(j)}$ such that
\begin{equation}\label{a2}
\widetilde{\BE}^{(j)}_{\tau} \Big[ \int_{\tau}^{T} \Big( \frac{h^{(j)}(\pi^{(-j,*)})}{1-\beta_j}\Big)^{+} \ds \Big]\leq
\widetilde{\BE}^{(j)}_{\tau} \Big[ \int_{\tau}^{T} \Big( \frac{h^{(j)}(\pi^{(-j,*)})}{1-\beta_j}\Big)^{+} \ds \Big]\leq c_3^{(j)}
\quad
\end{equation}
holds for any stopping times $\tau\leq T$. From Lemma \ref{Linear BSDE 1}, BSDE \eqref{4103} admits a solution in
$L^{\infty}_{\mathcal{F}}(0,T; \BR_{\gg1})\times L^{2,\: \mathrm{BMO}}_{\mathcal{F}} (0, T; \BR)$.
So BSDE \eqref{4101} admits a solution in $L^{\infty}_{\mathcal{F}}(0,T; \BR)\times L^{2,\: \mathrm{BMO}}_{\mathcal{F}} (0, T; \BR)$.

Applying It\^{o}'s formula to
$\big[X^{(j,*)}\big]^{\beta_j} \big[\widehat{X}^{(-j,*)}\big]^{\gamma_j}e^{P^{(j,*)}}$,
where $X^{(j,*)}$, $\widehat{X}^{(-j,*)}$ and $(P^{(j,*)},\Lambda^{(j,*)})$ are solutions of \eqref{203}, \eqref{4100} and \eqref{4101} with the Nash equilibrium given in \eqref{4203}, we get

\begin{equation*}
\begin{aligned}
\dd \frac{ [X^{(j,*)}]^{\beta_j} [\widehat{X}^{(-j,*)}]^{\gamma_j}e^{P^{(j,*)}} }
{ [X^{(j,*)}]^{\beta_j} [\widehat{X}^{(-j,*)}]^{\gamma_j}e^{P^{(j,*)}} }
= \Big( \beta_j\sigma\pi^{(j,*)} +\frac{\gamma_j\sigma}{n}\sum_{k\neq j} \pi^{(k,*)} +\Lambda^{(j,*)} \Big) \dw(s),\quad 1\leq j\leq n.
\end{aligned}
\end{equation*}
Thus, for all $1\leq j\leq n$, we have
\begin{equation*}
\begin{aligned}
&\big[X^{(j,*)}(t)\big]^{\beta_j} \big[\widehat{X}^{(-j,*)}(t)\big]^{\gamma_j} e^{P^{(j,*)}(t)}\\
=\,&\big[x^{(j)}\big]^{\beta_j} \big[\hat{x}^{(-j)} \big]^{\gamma_j} e^{P^{(j)}(0)}
\mathcal{E}\Big(\int_0^{t} \Big(\beta_j\sigma\pi^{(j,*)} +\frac{\gamma_j\sigma}{n}\sum_{k\neq j} \pi^{(k,*)} +\Lambda^{(j,*)}\Big) \dw(s) \Big).
\end{aligned}
\end{equation*}
Because for all $1\leq j\leq n$, $\beta_j\sigma\pi^{(j,*)} +\frac{\gamma_j\sigma}{n}\sum_{k\neq j} \pi^{(k,*)} +\Lambda^{(j,*)} \in L^{2,\: \mathrm{BMO}}_{\mathcal{F}}(0,T;\BR)$, $\big[X^{(j,*)}(t)\big]^{\beta_j} \big[\widehat{X}^{(-j,*)}(t)\big]^{\gamma_j} e^{P^{(j,*)}(t)}$ is a uniformly integrable martingale on $[0,T]$. Hence, we can conclude that the Nash equilibrium given in \eqref{4203} is admissible. We finish the proof.

\end{proof}

\begin{remark}
The Nash equilibrium in $\seta_2$ for game \eqref{203}-\eqref{204} can also be given by \eqref{4206} and
$(\widehat{P}^{(j)},\widehat{\Lambda}^{(j)})_{1\leq j \leq n}$ is the solution of $n$-dimensional coupled BSDE \eqref{4208}.
Fortunately, the combination $(\widehat{P}^{(j)}+ \delta_j\gamma_j\bar{\lambda}\overline{P},
\widehat{\Lambda}^{(j)}+ \delta_j\gamma_j\bar{\lambda}\overline{\Lambda})_{1\leq j\leq n}$ is the solution of $n$-dimensional decoupled BSDE,
which can also be regarded as solutions of $n$ $1$-dimensional BSDEs \eqref{4202}. We use \eqref{4203} in Theorem \ref{NE 2} because the solvability of $1$-dimensional BSDEs \eqref{4202} are easier to prove.
\end{remark}

\begin{remark}
If the interest rate, the appreciation rate and volatility rate of the risky security are all constants and $r=0$, $\mu>0$, $\sigma>0$, then the unique solution of BSDE \eqref{4202} is
$$(Y^{(j)},Z^{(j)})=\Big(\Big[\frac{1}{2}C_2^{(j)}\rho^2+C_3^{(j)}r\Big] (T-t),0\Big),~~ 1\leq j\leq n.$$
Therefore, when the market parameters are all constants, a Nash equilibrium in $\seta_2$ for game \eqref{203}-\eqref{204} is
\begin{equation*}
\begin{aligned}
\pi^{(j,*)}(t) = C_1^{(j)}\frac{ \rho }{\sigma},~~ 1\leq j\leq n.
\end{aligned}
\end{equation*}
This Nash equilibrium is the same as the one given in Corollary 3.2 in \cite{Zariphopoulou 1}. But we still cannot obtain the uniqueness of the Nash equilibrium in $\seta_2$ even the market parameters are all constants, because we cannot prove the solvability of BSDE \eqref{4101} with a strategy $(\pi^{(1)},\pi^{(2)}, \cdots,\pi^{(n)})\in\big(L^{2,\text{BMO}}_{\mathcal{F}}(0,T;\BR)\big)^n$. In fact, we can only get that this Nash equilibrium is unique in constant strategies, just like Corollary 3.2 in \cite{Zariphopoulou 1}.
\end{remark}

\section{Concluding remarks}\label{Conclusion}

In this paper, we studied a competitive optimal portfolio selection problem in a non-Markovian financial market for both CARA and CRRA utilities.
We succeeded in obtaining the Nash equilibria for these two types of utility via quadratic BSDE theory.
For the CARA utility case, we used the ABR method to get the Nash equilibrium in $\seta_1$, which has three cases depending on market and competition parameters.
For the CRRA utility case, we used the BR method to obtain a Nash equilibrium in $\seta_2$, which is related to a new kind of quadratic BSDEs with unbounded coefficients.
The solvability of this new kind of quadratic BSDEs with unbounded coefficients is established, which is interesting in its own right from the BSDE theory point of view.
%With the help of a decoupling technology, we can even give the limiting strategy for the CARA and CRRA utilities when the number of agent tends to be infinite; see our below argument.
We defined the admissible strategy set $\seta_2$ as the largest set that makes the game problem meaningful and solvable, but at the cost that we loss the uniqueness.
Is it possible to find a subset $\seta'_2$ of $\seta_2$ such that the Nash equilibrium in \eqref{4203} is unique in $\seta'_2$?
In order to answer this question, one has to resolve the uniqueness issue of the new kind of quadratic BSDE \eqref{4101} with unbounded coefficients, which is beyond our ability.

The games studied in this paper are of finite population.
Benefiting from a decoupling technique, the Nash equilibria we obtained for the CARA and CRRA utility cases are associated with two 1-dimensional BSDEs and $n$ 1-dimensional BSDEs, respectively. Under some mild assumptions, we can give the limiting strategy for the two cases when the number of agent tends to be infinite.

To see this, let the risk parameters and competition parameters of the whole population have two distributions.
Let $\bm{\delta}$ and $\bm{\theta}$ denote random variables with these two distributions.
When $\BE[\bm{\theta}]<1$, the limiting strategy for the CARA utility case is
\begin{equation*}
\begin{aligned}
\lim \limits_{n\rightarrow \infty }\pi^{(j,*)}(t) = -\frac{ \eta(t) }{\sigma(t)\psi(t)}X^{(j)}(t)
+\Big( \delta_j+ \frac{\BE[\bm{\delta}]\theta_j}{1-\BE[\bm{\theta}]} \Big)
\frac{ \rho(t)+\Delta(t)+\frac{\eta(t)}{\psi(t)} }{\sigma(t)\psi(t)},
\end{aligned}
\end{equation*}
and the limiting strategy for the CRRA utility case is
\begin{equation*}
\begin{aligned}
\lim \limits_{n\rightarrow \infty }\pi^{(j,*)}(t) = \frac{ K_1^{(j)}\rho(t)+Z^{(j)}(t) }{\sigma(t)},
\end{aligned}
\end{equation*}
where $(\psi,\eta)$ and $(\varphi,\Delta)$ are unique solutions of BSDE \eqref{3201} and \eqref{3202},
and $(Y^{(j)},Z^{(j)})$ is a solution of the following quadratic BSDE:
\begin{equation*}
\left\{
\begin{aligned}
\dd Y^{(j)}(t)&=-\Big[\frac{1}{2}\big[Z^{(j)}(t)\big]^2+(K_1^{(j)}-1)\rho Z^{(j)}(t)+\frac{1}{2}K_2^{(j)}\rho^2+K_3^{(j)}r \Big]\dt +Z^{(j)}(t)\dw(t),\\
Y^{(j)}(T)&=0
\end{aligned}
\right.
\end{equation*}
with
$$K_1^{(j)}\triangleq \frac{ \theta_j(1-\delta_j)\BE[\bm{\delta}] } { \BE[\bm{\theta}(1-\bm{\delta})] } +\delta_j,~~K_2^{(j)}\triangleq K_1^{(j)}(K_1^{(j)}-1),~~
K_3^{(j)}\triangleq K_1^{(j)}-1.$$

One of possible interesting extensions of this paper is to construct a mean field game formulation such that its mean field equilibria are consistent with the above limit strategies as in \cite{Zariphopoulou 1}.

%\appendix

\begin{appendix}

\section{The solvability of a new linear BSDE with unbounded coefficients}

\renewcommand*{\thetheorem}{\mbox{\textrm A.\arabic{theorem}}}

\begin{lemma} \label{Linear BSDE 1}
Given any $\BR$-valued $\{\mathcal{F}_t \}_{t\geq0}$-adapted process $a(\cdot)$ and
$b\in L^{2,\: \mathrm{BMO}}_{\mathcal{F}} (0, T; \BR)$, let
$$J(t)\triangleq\exp\bigg( \int_0^t a(s)\ds \bigg),\quad
N(t)\triangleq \mathcal{E}\bigg( \int_0^t b(s) \dw (s) \bigg).$$
Note that $N(t)$ is a uniformly integrable martingale, thus $ \overline{W}(t) \triangleq W(t) -\int_0^t b(s)\ds$ is a Brownian
motion under the probability measure $\overline{\BP}$ defined by
\begin{equation*}
%\label{B2}
\frac{\dd \overline{\BP}}{\dd \BP}\bigg|_{\mathcal{F}_T}=N(T).
\end{equation*}
Let $\overline{ \BE}_{t}$ be the conditional expectation w.r.t. the probability measure $\overline{\BP}$.
Let $\xi$ be an $\mathcal{F}_T$-measurable random variable such that
\begin{equation*}
%\label{appass0}
0<\underline{\xi}:=\essinf\xi \leq \overline{\xi}:=\esssup\xi<\infty.
\end{equation*}
Assume there are positive constants $c_{1},c_{2}$ and $c_{3}$ such that
\begin{equation}\label{appass1}
%C_{1}\leq \xi\leq C_{2},~~~
c_{1}\leq \overline{ \BE}_t\bigg[\exp\bigg( \int_t^T a(s)\ds \bigg)\bigg]\leq c_{2},~~\forall\;t\in[0,T],
\end{equation}
and
\begin{equation}\label{appass2}
\overline{\BE}_{\tau} \Big[ \int_{\tau}^{T} a(s)^{+} \ds \Big]\leq c_{3}
\end{equation}
holds for any stopping time $\tau\leq T$.
%
%holds for all $t\in[0,T]$,
%
%If there exists a positive constant $\epsilon\in(0,\frac{1}{2})$ such that for all $\{\mathcal{F}_t\}_{t\geq0}$-stopping times $\tau\leq T$,
%$$\overline{\BE} \Big[ \int_{\tau}^T |a(s)| \ds \;\Big| \;\mathcal{F}_{\tau} \Big]\leq \frac{1}{2}-\epsilon,$$
%where $\overline{ \BE}$ is the expectation w.r.t. the probability measure $\overline{\BP}$ defined by \eqref{B2}.
%which guarantees $\sqrt{|a(\cdot)|}\in L^{2,\: \mathrm{BMO}}_{\mathcal{F}} (0, T; \BR_{\gg1})$ and
%$$ \Big\|\int_0^{\cdot} \sqrt{|a(s)|} \dd \overline{W}(s)\Big\|_{\mathrm{BMO}_2}
% =\sup\limits_{\tau\leq T}\Big( \operatorname*{ess\,sup}
% \overline{\BE}\Big[ \int_{\tau}^T |a(s)| \,\ds\;\Big| \;\mathcal{F}_{\tau} \Big] \Big)^{\frac{1}{2}}
%\leq \sqrt{\frac{1}{2}-\epsilon} <\frac{1}{\sqrt{2}},$$
Then the following 1-dimensional linear BSDE
\begin{equation}\label{B3}
\left\{
\begin{aligned}
\dd Y(t)&=-\big[a(t)Y(t)+b(t)Z(t)\big]\dt +Z(t)\dw(t),\\
Y(T)&=\xi,
\end{aligned}
\right.
\end{equation}
admits a solution $(Y,Z)\in L^{\infty}_{\mathcal{F}}(0,T; \BR_{\gg1})\times L^{2,\: \mathrm{BMO}}_{\mathcal{F}} (0, T; \BR)$. Moreover, $c_{1}\underline{\xi} \leq Y \leq c_{2} \overline{\xi}$.
% and
%$ \widetilde{C}e^{-\frac{1}{2}}\leq Y\leq \widetilde{C} \big/ \big(1-\frac{1}{\sqrt{2}}\big)$.
%
If we further assume
\begin{equation}\label{appass3}
\overline{ \BE}\bigg[\exp\bigg( 2\int_0^T a(s)\ds \bigg)\bigg]<\infty,
\end{equation}
then \eqref{B3} admits a unique solution
$(Y,Z)\in L^{\infty}_{\mathcal{F}}(0,T; \BR_{\gg1})\times L^{2,\: \mathrm{BMO}}_{\mathcal{F}} (0, T; \BR)$ such that
\begin{equation}\label{appass4}
\overline{\BE}\bigg[ \int_0^{T} \big|J(s)Z(s)\big|^2 \ds \bigg] < \infty.
\end{equation}
%\red{ the uniqueness may not be needed in the result.} \blue{Yes.}
\end{lemma}
\begin{proof}
We set
$$
Y (t) = \big[J(t)\big]^{-1}\overline{ \BE}_t\big[\xi J(T) \big]= \overline{ \BE}_t\bigg[\xi\exp\bigg( \int_t^T a(s)\ds \bigg)\bigg].$$
Then by \eqref{appass1}, we have $Y (T) = \xi $ and
$c_{1}\underline{\xi} \leq Y (t) \leq c_{2} \overline{\xi}$,
% Using Jessen inequality and Lemma \ref{Inequality 1}, for any $t\in [0,T]$, we get
% \begin{equation*}
%\begin{aligned}
% \xi e^{-\frac{1}{2}} \leq \xi e^{-\overline{ \BE}_t\big[\int_t^T |a(s)|\ds\big] }
% \leq \xi \overline{ \BE}_t\big[e^{-\int_t^T |a(s)|\ds}\big] \leq
% Y (t)
% %= \xi \overline{ \BE}_t\big[e^{\int_t^T a(s)\ds}\big]
% \leq \xi \overline{ \BE}_t\big[e^{\int_t^T |a(s)|\ds}\big]
%\leq \xi \Big/ \Big(1-\frac{1}{\sqrt{2}}\Big).
%\end{aligned}
%\end{equation*}
so $Y\in L^{\infty}_{\mathcal{F}}(0,T; \BR_{\gg1})$. \footnote{The requirement \eqref{appass1} is equivalent to $Y\in L^{\infty}_{\mathcal{F}}(0,T; \BR_{\gg1})$.}
%Similarly, we have
%$$\overline{ \BE}\big[ J(T)^2 \big]
%\leq \overline{ \BE} \big[e^{\int_0^T 2|a(s)|\ds}\big]\leq \frac{1} { 1-\sqrt{2}\overline{a} } .$$
% Thus,
% $J(t) Y (t) =\overline{ \BE}_t\big[\xi J(T) \big]$
% is a square integrable martingale under $\overline{\BP}$. By the martingale presentation theorem, there exists $\widehat{Z}\in L^{2}_{\mathcal{F}}(0,T;\BR)$

Because $J(t) Y (t) =\overline{ \BE}_t\big[Y(T) J(T) \big]$ is a martingale under $\overline{\BP}$, so there exists a unique
$\widehat{Z}\in L^{2,\text{loc}}_{\mathcal{F}}(0,T;\BR)$ under $\overline{\BP}$ such that
$$J(t)Y (t)= J(0)Y (0) +\int_0^t \widehat{Z}(s)\dd \overline{W} (s).$$
As a consequence, we get $\dd \,(J(t)Y (t)) = \widehat{Z}(t)\dd \overline{W} (t)$.
Let $Z(t) = {[J(t)]}^{-1}\widehat{Z}(t)$.
By It$\hat{\text{o}}$'s lemma,
\begin{equation*}
\begin{aligned}
\dd Y (t) = \dd \big[ J(t)^{-1} \cdot \big(J(t)Y (t)\big)\big]
=-\big[a(t)Y (t)+b(t)Z(t)\big]\dt +Z(t) \dw (t).
\end{aligned}
\end{equation*}
Thus, $(Y, Z)$ satisfies \eqref{B3}.

As ${[J(\cdot)]}^{-1}$ is a continuous process, it is almost surely bounded on $[0,T]$.
Then the stochastic integral $\int_0^{\cdot} Y(t)Z(t) \dd \overline{W}(t)$ is a local martingale.
%Let $\tau_k$ be a localizing sequence of stopping times for it.
%Applying It\^{o}'s formula to ${Y}^2$,
%\begin{align*}
%\overline{ \BE} \Big[ \int_{0}^{T\wedge \tau_k} {|Z(s)|}^2 \ds \Big]
%&= \overline{ \BE} \big[ Y(T\wedge \tau_k)^2 \big]- {Y(0)}^2
% + \overline{\BE} \Big[ \int_{0}^{T\wedge \tau_k} 2a(s)Y(s)^2 \ds \Big]\\
%&\leq \xi ^{2}C_{2}^{2}+2\xi C_{2}\overline{\BE} \Big[ \int_{0}^{T} a(s)^{+} \ds \Big].
%\end{align*}
%Sending $k\to\infty$, we obtain from the monotone and dominated convergence theorems,
%\begin{align*}
%\overline{ \BE} \Big[ \int_{0}^{T} {|Z(s)|}^2 \ds \Big]
%&\leq \xi ^{2}C_{2}^{2} +2\xi ^{2}C_{2}^{2}\overline{\BE} \Big[ \int_{0}^{T} a(s)^{+} \ds \Big],
%\end{align*}
%so $\widehat{Z}\in L^{2}_{\mathcal{F}}(0,T;\BR)$.
%
For any given stopping time $\tau\leq T$, let
\[\tau_{k}=\inf\bigg\{t\in[\tau,T]: \int_{\tau}^{t} |Y(s)Z(s)|^{2}\ds\geq k\bigg\}\wedge T,\]
with the conventioin $\inf\emptyset\triangleq +\infty$. Then, $\{\tau_k\}_{k\geq 1}$ is an increasing sequence of stopping times satisfying
$\tau_k\rightarrow T$ as $k\rightarrow \infty$.
%such that, for all $k$ and any stopping time $\tau\leq (T\wedge \tau_k)$,
%$$\textcolor{red}{\overline{ \BE}_{\tau} \Big[\int_{\tau}^{T\wedge \tau_k} Y(t)Z(t) \dd \overline{W}(t) \Big]=0.}$$
%\footnote{Can this equation hold?}
Applying It\^{o}'s formula to ${Y}^2$, we get
\begin{align*}
\overline{ \BE}_{\tau} \Big[ \int_{\tau}^{ \tau_k} {|Z(s)|}^2 \ds \Big]
&= \overline{ \BE}_{\tau} \big[ Y( \tau_k)^2 \big]- {Y(\tau)}^2
+ \overline{\BE}_{\tau} \Big[ \int_{\tau}^{ \tau_k} 2a(s)Y(s)^2 \ds \Big]\\
%\leq 2\Big[\xi \Big/ \Big(1-\frac{1}{\sqrt{2}}\Big)\Big]^2.\\
&\leq c_{2}^{2}\overline{\xi}^{2}+2c_{2}^{2}\overline{\xi}^{2}
\overline{\BE}_{\tau} \Big[ \int_{\tau}^{ \tau_k} a(s)^{+} \ds \Big]\\
&\leq c_{2}^{2}\overline{\xi}^{2}+2c_{2}^{2}c_{3}\overline{\xi}^{2}.
\end{align*}
Sending $k\to\infty$, using monotone convergence theorem, we obtain
\begin{align*}
\overline{ \BE}_{\tau} \Big[ \int_{\tau}^{T} {|Z(s)|}^2 \ds \Big]
&\leq c_{2}^{2}\overline{\xi}^{2}+2c_{2}^{2}c_{3}\overline{\xi}^{2},
\end{align*}
which yields that $\int_0^{\cdot} Z(s) \dd \overline{W} (s)$ is a BMO martingale under $\overline{\BP}$.
Consequently, $\int_0^{\cdot} Z(s) \dw (s)$ is a BMO martingale under $\BP$.
This shows that $(Y, Z)$ is a solution of \eqref{B3} in $L^{\infty}_{\mathcal{F}}(0,T; \BR_{\gg1})\times L^{2,\: \mathrm{BMO}}_{\mathcal{F}} (0, T; \BR)$.

Now we further assume the condition \eqref{appass3} holds, then $J(t) Y (t)$ defined in above is a square integrable martingale, so that the unique process $\widehat{Z}\in L^{2}_{\mathcal{F}}(0,T;\BR)$ under $\overline{\BP}$, that is \eqref{appass4} holds.

To prove the uniqueness, suppose $(Y, Z)$ and $(\widetilde{Y}, \widetilde{Z})$ are both solutions of \eqref{B3} in $L^{\infty}_{\mathcal{F}}(0,T; \BR_{\gg1})\times L^{2,\: \mathrm{BMO}}_{\mathcal{F}} (0, T; \BR)$ that satisfy \eqref{appass4}. Set
$$ \overline{Y}\triangleq Y-\widetilde{Y},\quad \overline{Z}\triangleq Z-\widetilde{Z}.$$
Then $(\overline{Y},\overline{Z})$ satisfies the following BSDE:
\begin{equation*}
\left\{
\begin{aligned}
\dd \overline{Y}(t)&=-\big[a(t)\overline{Y}(t)+b(t)\overline{Z}(t)\big]\dt +\overline{Z}(t)\dw(t),\\
\overline{Y}(T)&=0.
\end{aligned}
\right.
\end{equation*}
By It\^{o}'s lemma, it follows
$$J(t)\overline{Y}(t)= -\int_t^T J(s)\overline{Z}(s)\dd \overline{W} (s).$$
Note the condition \eqref{appass4} implies $\int_0^{\cdot} J(s)\overline{Z}(s)\dd \overline{W} (s)$ is a square integrable martingale under $\overline{\BP}$, so
$$J(t)\overline{Y}(t)= -\overline{\BE}_t \Big[\int_t^T J(s)\overline{Z}(s)\dd \overline{W} (s)\Big]=0,\quad \forall \,t\in[0,T].$$
%
%Note $\overline{Z}$ is also a BMO under $\overline{\BP}$.
%, so for any $p>1$,
%$$\overline{\BE}\bigg[ \Big(\int_0^{T} \big|\overline{Z}(s)\big|^2 \ds\Big)^{p}\bigg]
%<\infty.$$
%If for some $q>0$ such that
%$$\overline{\BE}\bigg[ \Big(\int_0^{T} J(s)^2 \ds\Big)^{q}\bigg] <\infty,$$
%then by H\^{o}lder's inequality gives
%\begin{align*}
%\overline{\BE}\Big[ \int_0^{T} \big|J(s)\overline{Z}(s)\big|^2 \ds \Big]
%&\leq \overline{\BE}\bigg[ \Big(\int_0^{T} J(s)^{2q} \ds\Big)^{1/q}\Big(\int_0^{T} \big|\overline{Z}(s)\big|^{2p} \ds\Big)^{1/p}\bigg]\\
%&\leq
%\end{align*}
%
%
%The stochastic integral in above equation is a local martingale. By dominated convergence theorem,
%for any $t\in[0,T]$, we get $\overline{Y}(t)=0$ by taking conditional expectation $\overline{\BE}_t$ on both sides.
%
Since $J>0$, we get $\overline{Y}=0$. And consequently,
$$\overline{\BE}\Big[ \int_0^{T} \big|J(s)\overline{Z}(s)\big|^2 \ds \Big]
= \overline{\BE}\Big[ \Big( \int_0^{T}J(s)\overline{Z}(s) \dd \overline{W} (s) \Big)^2\Big]
=\big[J(0)\overline{Y}(0) \big]^2=0.$$
It follows $\overline{Z} = 0$ as $J>0$. This completes the proof.
\end{proof}%\begin{lemma} \label{Linear BSDE 2}
%Let
%$f$ be a $\BR_{\gg1}$-valued $\{\mathcal{F}_t\}_{t\geq0}$-adapted process and
%$\sqrt{f}\in L^{2,\: \mathrm{BMO}}_{\mathcal{F}} (0, T; \BR_{\gg1})$.
%For any $b\in L^{2,\: \mathrm{BMO}}_{\mathcal{F}} (0, T; \BR)$, $\widetilde{C} \in \BR$ and $\widetilde{C} \geq 0$,
%the following 1-dimensional linear BSDE
%\begin{equation}\label{B5}
%\left\{
%\begin{aligned}
%\dd Y(t)&=-\big[b(t)Z(t)+f(t)\big]\dt +Z(t)\dw(t),\\
%Y(T)&=\widetilde{C},
%\end{aligned}
%\right.
%\end{equation}
%admits a unique solution $(Y,Z)\in L^{\infty}_{\mathcal{F}^W}(0,T; \BR)\times L^{2,\: \mathrm{BMO}}_{\mathcal{F}} (0, T; \BR)$ and
%$$ \widetilde{C}\leq Y(t)= \overline{ \BE}_t \Big[ \widetilde{C}+\int_t^T f(s)\ds \Big]
% \leq \widetilde{C}+ \Big\|\int_0^{\cdot} \sqrt{f(s)} \dd \overline{W}(s)\Big\|^2_{\mathrm{BMO}_2} .$$
%\end{lemma}
%\begin{proof}
% The proof of this lemma is similar to Lemma \ref{Linear BSDE 1} and will not be repeated here.
%\end{proof}

\end{appendix}

%\noindent\textbf{References }

\end{document}